\documentclass[12pt]{amsart}
\usepackage{amsmath,amsthm,amsfonts,amssymb} 
\usepackage {textcomp, dsfont, enumitem}
\usepackage[all]{xy}
\usepackage{color}

\allowdisplaybreaks 

\numberwithin{equation}{section} 
\newtheorem{thm}[equation]{Theorem} 
\newtheorem{cond}[equation]{Conditions}

\newtheorem{lemma}[equation]{Lemma} 

\newtheorem{example}[equation]{Example}
\newtheorem{remark}[equation]{Remark}

\newtheorem*{thm*}{Theorem}

\DeclareMathAlphabet{\mathpzc}{OT1}{pzc}{m}{it}

\DeclareMathOperator{\HH}{HH}

\DeclareMathOperator{\Ext}{Ext}
\DeclareMathOperator{\Hom}{Hom}

\DeclareMathOperator{\ch}{char}

\newcommand{\B}{\mathbb B}
\newcommand{\K}{\mathbb K}

\newcommand{\DOT}{\setlength{\unitlength}{1pt}\begin{picture}(2.5,2)
               (1,1)\put(2,3.5){\circle*{3}}\end{picture}}

\newcommand{\Z}{{\mathbb Z}}
\newcommand{\N}{{\mathbb N}}

\newcommand{\HHD}{{\rm HH}^{\DOT}}

\newcommand{\q}{\mathbf{q}}
\newcommand{\m}{\mathbf{m}}

\begin{document}


\title[Hochschild cohomology of group extensions]
{Hochschild cohomology of group extensions of quantum complete intersections}



\author{Lauren Grimley} 
\address{Department of Mathematics\\Texas A\&M University \\ College Station, TX 77843}
\email{lgrimley@math.tamu.edu}


\date{May 17, 2016}

\begin{abstract}
We formulate the Gerstenhaber algebra structure of Hochschild cohomology of finite group extensions of some quantum complete intersections.  When the group is trivial, this work characterizes the graded Lie brackets on Hochschild cohomology of these quantum complete intersections, previously only known for a few cases.  As an example, we compute the algebra structure for two generator quantum complete intersections extended by select groups.
\end{abstract}

\maketitle


\section{Introduction}

The behavior of Hochschild cohomology of quantum complete intersections has been shown to vary greatly based on the choice of quantum coefficients \cite{BE, BGMS, O}, exhibiting behaviors unseen with commutative algebras.  The quantum complete intersections of interest in this document are noncommutative generalizations of truncated polynomial rings.  Formally, if $k$ is a field, $m_1, m_2, .... , m_n$ are positive integers, and $\q=\{q_{i,j}\}_{i,j \in \{1,2,...,n\}}$ such that $q_{i,j} \in k^*$ for all $i,j \in \{1,2, ..., n\}$, $q_{j,i}=q_{i,j}^{-1}$, and $q_{i,i}=-1$, then $$\Lambda_{\q}^{\m} = k \langle x_1, x_2, ..., x_n | ~x_i x_j = - q_{i,j}x_j x_i, ~x_i^{m_i}=0 \textrm{ for all } i,j \in \{1,2, ... n\}\rangle$$ is the quantum complete intersection on $\q$ and $\m=(m_1, m_2, ..., m_n)$.  


In \cite{H}, Happel asked: does the vanishing of Hochschild cohomology in high degrees of a finite dimensional algebra over a field imply finite global dimension?  Avramov and Iyengar showed  in \cite{AI} that, for commutative algebras, finite Hochschild cohomology does imply finite global dimension.  Quantum complete intersections, however, provided the first examples for which this implication was shown not to hold.  In \cite{BGMS}, Buchweitz, Green, Madsen and Solberg showed that the two generator quantum complete intersections, $\Lambda_{\q}^{(2,2)}$, with $\q$ not a root of unity have infinite global dimension and finite Hochschild cohomology.  When the quantum coefficient is a root of unity, $\Lambda_{\q}^{(2,2)}$ has infinite Hochschild cohomology with large gaps determined by the root of unity \cite{BGMS}.  Bergh and Erdmann generalized this result in \cite{BE}, showing that the two generator quantum complete intersections, $\Lambda_{\q}^{(m_1,m_2)}$, have finite Hochschild cohomology if and only if the quantum coefficient is not a root of unity.  In \cite{O}, Oppermann furthered this study by showing that the Hochschild cohomology of finite quantum complete intersections, $\Lambda_{\q}^{(m_1,m_2, ..., m_n)}$, is intimately tied to the choice of quantum coefficients.  


While the algebra structure of Hochschild cohomology of quantum complete intersections with respect to the cup product has been studied, Hochschild cohomology of an associative algebra also admits a graded Lie bracket which is less understood.  The graded Lie algebra structure on Hochschild cohomology has been studied for monomial algebras \cite{CGX, SF, Stra}, skew group algebras \cite{SW}, tensor products \cite{LZ}, group extensions of polynomial rings \cite{NW2} and skew polynomial rings \cite{WZ}, and the quantum complete intersections $\Lambda_{\q}^{(2,2)}$ \cite{GNW}.  Most recently in \cite{BKL}, Benson, Kessar, and Linkelman studied the Lie algebra structure of the first Hochschild cohomology $k$-module of $\Lambda_{\q}^{(p,p)}$ for a prime $p$ and $\q$ of order dividing $p-1$.

In this paper, we focus on $\Lambda_{\q}^{(2,2,...,2)}$ and allow a finite group, $G$, to act diagonally on the basis $\{x_1, x_2, ..., x_n\}$ in order to study Hochschild cohomology of the corresponding skew group algebra.  Hochschild cohomology in this case is controlled by the quantum coefficients as well as the choice of group action.  We study the Gerstenhaber algebra structure of Hochschild cohomology of group extensions of these quantum complete intersections, providing explicit formulas for the cup product and graded Lie bracket which encode the noncommutative behaviors.  When the group is trivial, this work characterizes the graded Lie algebra structure of Hochschild cohomology of the quantum complete intersections $\Lambda_{\q}^{(2,2,...,2)}$, extending the previous results of \cite{BE, BGMS, O}.

We utilize the notion of twisted tensor products given by Bergh and Oppermann in \cite{BO} and adapt the technique of Wambst in \cite{W} to compute the Hochschild cohomology in Section \ref{vs}.  As in \cite{W}, we restrict to the $\ch k = 0$ case as it is required by the contracting homotopy used in the proof.  We also restrict to diagonal actions because the deconstruction used does not generalize immediately to nondiagonal actions.  We use techniques adapted from \cite{BGMS} to compute the cup product in Section \ref{cup}.  Using an alternative bracket description given by Negron and Witherspoon in \cite{NW} and \cite{NW2}, we compute the bracket structure in Section \ref{bracket}.  The bracket definition given in \cite{NW} works well with Koszul algebras.  We focus on $\Lambda_{\q}^{(2,2,...,2)}$ because it is the exclusive case for which $\Lambda_{\q}^{\m}$ is Koszul.  As an example, these structures are computed for several cases when $n=2$ in Section \ref{ex} and \ref{equal}.  These example cases are generalizations of, and can be compared to, the examples in \cite[Section 5]{GNW}, \cite[Section 3]{BGMS}, and \cite[Section 3]{BE} when the characteristic of the field is $0$.



For the remainder of the paper, we assume $k$ is a field of characteristic $0$.  Unless otherwise noted, $\otimes=\otimes_k$ and all modules are left modules.


\section{Preliminaries}\label{preliminaries}

Let $A$ be a $k$-algebra.  Let $A^e=A \otimes A^{op}$ be the \textsl{enveloping algebra} of $A$ where $A^{op}$ is the opposite algebra.  Because $A$ is a $k$-algebra, the \textsl{Hochschild cohomology} of $A$ is $$\HH^*(A) = \Ext^*_{A^e}(A, A).$$  Hochschild cohomology of $A$ is a \textsl{Gerstenhaber algebra}, having a graded commutative algebra structure given by the cup product, $\smile$, and a graded Lie bracket, $[-,-]$, also called a Gerstenhaber bracket, such that $[-, f]$ is a graded derivation with respect to $\smile$.  



Any $k$-algebra $A$ has a projective $A^e$-module resolution given by the bar resolution.  The \textsl{bar resolution} of $A$, $\mathbb{B}(A)$, is given by
$$\mathbb{B}(A): ...\xrightarrow{\delta_3} A^{\otimes 4} \xrightarrow{\delta_2} A \otimes A \otimes A \xrightarrow{\delta_1} A \otimes A \xrightarrow{\delta_0} A  \rightarrow 0$$
where
$$d_n(a_0 \otimes a_1 \otimes ... \otimes a_{n+1}) = \sum_{m=0}^{n} (-1)^m a_0 \otimes a_1 \otimes ... \otimes a_m a_{m+1} \otimes ... \otimes a_{n+1} $$ 
for $a_0, ..., a_{n+1} \in A$.



Recall the bar resolution admits a comultiplication given by the \textsl{diagonal map} on $\mathbb{B}(A)$, $\Delta_{\mathbb{B}(A)} : \mathbb{B}(A) \rightarrow \mathbb{B}(A) \otimes_A \mathbb{B}(A)$ which is a chain map given by 
$$\Delta_{\mathbb{B}(A)} (a_0 \otimes a_1 \otimes ... \otimes a_{n+1}) = \sum_{m=0}^n (a_0 \otimes ... \otimes a_m \otimes 1) \otimes_A (1 \otimes a_{m+1} \otimes ... \otimes a_{n+1}) $$
for $a_0, ..., a_{n+1} \in A$.  We define the cup product, $\smile$, on the chain level in terms of this diagonal map.  

Let $f \in \Hom_{A^e}((\mathbb{B}(A))_n, A)$ and $g \in \Hom_{A^e}((\mathbb{B}(A))_m, A)$.  The cup product, $f \smile g \in \Hom_{A^e}((\mathbb{B}(A))_{n+m}, A)$, is the composition 
$$f \smile g: \mathbb{B}(A) \xrightarrow{\Delta_{\mathbb{B}(A)}} \mathbb{B}(A) \otimes_A \mathbb{B}(A) \xrightarrow{f \otimes_A g} A \otimes_A A \xrightarrow{\mu} A$$
where $\mu$ is the multiplication map given by $A \otimes_A A \cong A$.

The graded Lie bracket on Hochschild cohomology, $[-,-]$, was originally given by Gerstenhaber in \cite{G} for functions defined at the chain level on the bar resolution.  Here we use a construction given in \cite{NW} to define the Gerstenhaber bracket on any projective resolution, $\mathbb{K}$, of $A$ satisfying the following conditions:  
\begin{cond} \label{cond}
\begin{enumerate}[label=(\alph*)]
\item There is a chain map $\iota: \mathbb{K} \rightarrow \mathbb{B}(A)$ lifting $\mathds{1}_A$.
\item There is a chain map $\pi: \mathbb{B}(A) \rightarrow \mathbb{K}$ such that $\pi \iota = \mathds{1}_{\mathbb{K}}$.
\item $\mathbb{K}$ admits a diagonal map, $\Delta_{\mathbb{K}}: \mathbb{K} \rightarrow \mathbb{K} \otimes_A \mathbb{K}$, such that $\Delta_{\mathbb{B}(A)}\iota = (\iota \otimes_A \iota)\Delta_{\mathbb{K}}$.
\end{enumerate}
\end{cond}
\noindent Notice $\K= \B(A)$ trivially satisfies conditions above.  It is argued in \cite{NW} that if $A$ is a Koszul algebra and $\K$ is its Koszul resolution, then $\K$ satisfies the conditions above.

Let $\mathbf{m}: \mathbb{K} \rightarrow A$ be a quasi-isomorphism then, as in \cite[Section 3.2]{NW}, define $$F_{\mathbb{K}}=(\mathbf{m} \otimes_A \mathds{1}_{\K} - \mathds{1}_{\K} \otimes_A \mathbf{m}): \K \otimes_A \K \rightarrow \K.$$  Assume $\phi: \K \otimes_A \K \rightarrow \K$ satisfies $d_{\Hom(\K \otimes \K, \K)}(\phi) = F_{\K}$ and $f \in \Hom_{A^e}((\K)_n, A)$ and $g \in \Hom_{A^e}((\K)_m, A)$.  Define the $\circ$-product on the chain level as the composition 
\begin{align} \label{circdef} f \circ g : \K \xrightarrow{\Delta_{\K}} \K \otimes_A \K \xrightarrow{\Delta_{\K} \otimes_A \mathds{1}_{\K}} \K \otimes_A \K \otimes_A \K \xrightarrow{\mathds{1}_{\K} \otimes_A g \otimes_A \mathds{1}_{\K}} \K \otimes_A \K \xrightarrow{\phi} \K \xrightarrow{f} A \end{align}
where $\mathds{1}_{\K} \otimes_A g \otimes_A \mathds{1}_{\K}$ has Koszul signs.  That is, for $x_1 \otimes_A x_2 \otimes x_3 \in \K \otimes_A \K \otimes_A \K$ with $x_1$ homogeneous, $$(\mathds{1}_{\K} \otimes_A g \otimes_A \mathds{1}_{\K}) (x_1 \otimes_A x_2 \otimes_A x_3) = (-1)^{m ||x_1||} x_1 \otimes_A g(x_2) \otimes_A x_3$$ where $||x_1||$ is the homological degree of $x_1$.  Finally, define the bracket, $[f,g]$, on the chain level to be $$[f,g]=f \circ g - (-1)^{(n-1)(m-1)} g \circ f.$$  By \cite[Theorem 3.2.5]{NW}, this bracket induces the Gerstenhaber bracket on Hochschild cohomology.

We complete the preliminaries section by defining another crucial concept for this paper, the twisted tensor product.  Let $R$ and $S$ be $k$-algebras, graded by abelian groups $A$ and $B$ respectively.  Let $$t: A \otimes_{\Z} B \rightarrow k^*$$ be a homomorphism of abelian groups which we will call the \textsl{twisting map}.  We will denote $t(a \otimes_{\Z} b)= t^{<a|b>}$ for all $a \in A$ and $b \in B$.  Then, as in \cite{BO}, define $R \otimes^t S$, the \textsl{twisted tensor product} of $R$ and $S$, to be $R \otimes S$ as a vector space with multiplication 
$$(r_0 \otimes s_0)(r_1 \otimes s_1) = t^{<|r_1|||s_0|>} (r_0 r_1 \otimes s_0 s_1)$$
for all homogeneous elements $r_0, r_1 \in R$ and $s_0, s_1 \in S$, where $|\cdot|$ denotes the grading degree of the element.  



\section{Hochschild cohomology of group extensions of quantum complete intersections} \label{HH}


Let $\q=\{q_{i,j}\}_{i,j \in \{1,2,...,n\}}$ such that $q_{i,j} \in k^*$ for all $i,j \in \{1,2, ..., n\}$, $q_{j,i}=q_{i,j}^{-1}$, and $q_{i,i}=-1$.  Define $$\Lambda_{\q} = k \langle x_1, x_2, ..., x_n | x_i x_j = - q_{i,j}x_j x_i, x_i^2=0 \textrm{ for all } i,j \in \{1,2, ... n\} \textrm{ with } i<j\rangle.$$  Notice that this definition agrees with the definition of $\Lambda_{-\q}^{(2,2,...,2)}$ found in \cite{O}.  Let $G$ be a finite group which acts diagonally on the generating set $x_1, x_2, ..., x_n$ of $\Lambda_{\q}$.  For $g \in G$ and $\lambda \in \Lambda_{\q}$, let $^g \lambda$ denote the action of $g$ on $\lambda$.  Then for each $i \in \{1,2,...,n\}$ and $g \in G$, there exists a $\chi_{g,i} \in k$ such that $^g x_i = \chi_{g,i} x_i$.  Extend linearly to induce an action of $G$ on $\Lambda_{\q}$ by automorphisms.  

By restricting to considering only diagonal actions, we impart no additional conditions on $\q$, allowing us to consider group extensions of all quantum complete intersections in the form of $\Lambda_{\q}$.  See \cite[Lemma 4.2]{NW2} for the restrictions on $\q$ under more general actions.  Additionally, because we are assuming that $G$ is a finite group, notice $\chi_{g,i}$ is necessarily a root of unity for all $g \in G$ and $i \in \{1, 2, ..., n\}$.

Define $\Lambda_{\q} \rtimes G$, the \textsl{group extension} of $\Lambda_{\q}$ by $G$ (or \textsl{skew group algebra}), to be $\Lambda_{\q} \otimes kG$ as a vector space with multiplication determined by
$$(\lambda_0 \otimes g_0)(\lambda_1 \otimes g_1) = \lambda_0(^{g_0} \lambda_1) \otimes g_0 g_1$$
for $\lambda_0, \lambda_1 \in \Lambda_{\q}$ and $g_0, g_1 \in G$.  Notice $\Lambda_{\q} \rtimes 1 \cong \Lambda_{\q}$.


Because $\ch k \nmid |G|$, the Hochschild cohomology of $\Lambda_{\q} \rtimes G$ has a particular form, $$\HH^*(\Lambda_{\q} \rtimes G) \cong (\HH^*(\Lambda_{\q}, \Lambda_{\q} \rtimes G))^G,$$ the $G$-invariants of $\HH^*(\Lambda_{\q}, \Lambda_{\q} \rtimes G)$.  The more general result, $\HH^*(A \rtimes G) \cong (\HH^*(A, A \rtimes G))^G,$ for a $k$-algebra $A$ with $\ch k \nmid |G|$ was first given for Hochschild homology in \cite{L} and for the more general setting of Hochschild cohomology of smash product algebras in \cite{S}.  

With this result in mind, we use the remainder of this section to construct a resolution of $\Lambda_{\q}$ using the twisted tensor products of \cite{BO}.  First, notice we can construct $\Lambda$ by taking an iteration of twisted tensor products.  Consider $R_{x_i}=k\langle x_i \rangle/(x_i^2)$ which is $\Z$-graded by the degree of $x_i$ for $i \in \{1,2,...,n\}$.  Let $t_i^{<[j]|1>}=-q_{j,i+1}^{-1}$ for $i \in \{1,2,...,n-1\}$, $j<i$, and $[j]=(0,...,0,1,0,...,0)$, the $i$-tuple with a $1$ in the $j$th coordinate and $0$ otherwise.  Then $$\Lambda_{\q} \cong (...((R_{x_1} \otimes^{t_1} R_{x_2}) \otimes^{t_2} R_{x_3}) \otimes^{t_3} ...) \otimes^{t_{n-1}} R_{x_n}.$$


For each $i \in \{1,2,...,n\}$, $R_{x_i}$ has a projective resolution as an $R_{x_i}^e$-module given by $$\mathbb{K}_{x_i}: ...\xrightarrow{\delta_3} R_{x_i}^e \xrightarrow{\delta_2} R_{x_i}^e \xrightarrow{\delta_1} R_{x_i}^e \xrightarrow{\mu} R_{x_i}  \rightarrow 0$$ \noindent
where $\mu$ is multiplication, and $\delta_m^{x_i}(1 \otimes 1)=x_i \otimes 1 + (-1)^m 1 \otimes x_i$.  The boundary maps $\delta_m$ have degree 1.  Thus to form a graded resolution, we introduce degree shifts to get $$\mathbb{K}_{x_i}': ...\xrightarrow{\delta_3} R_{x_i}^e\langle 2 \rangle \xrightarrow{\delta_2} R_{x_i}^e \langle 1 \rangle \xrightarrow{\delta_1} R_{x_i}^e \xrightarrow{\mu} R_{x_i}  \rightarrow 0$$ where $(R_{x_i}^e \langle a \rangle)_b=(R_{x_i}^e)_{b-a}$, a graded $R_{x_i}^e$-module resolution of $R_{x_i}$.  By \cite[Lemmas 4.3, 4.4, 4.5]{BO}, $$\mathbb{K}=\textrm{Tot}(\mathbb{K}_{x_1}' \otimes^{t_1} \mathbb{K}_{x_2}' \otimes^{t_2} ...\otimes^{t_{n-1}} \mathbb{K}_{x_n}')$$ is a graded projective resolution of $\Lambda_{\textbf{q}}$ as a $(\Lambda_{\textbf{q}})^e$-module.  By \cite[Lemma 4.3]{BO}, we can rearrange the elements of $\K$ using the isomorphism of graded $(R_{x_j} \otimes^{t_j} R_{x_{j+1}})^e$-modules, $R_{x_j}^e \langle i_j \rangle \otimes^{t_j} R_{x_{j+1}}^e \langle i_{j+1} \rangle \cong (R_{x_j} \otimes^{t_j} R_{x_{j+1}})^e \langle i_j, i_{j+1} \rangle$ given by $$(r_1 \otimes r'_1) \otimes^{t_j} (r_2 \otimes r'_2) \mapsto t_j^{-<r'_1|r_2>-<i_j|r_2>-<r'_1|i_{j+1}>}(r_1 \otimes^{t_j} r_2) \otimes (r'_1 \otimes^{t_j} r'_2)$$ for $r_1,r'_1 \in R_{x_j}$ and $r_2, r'_2 \in R_{x_{j+1}}$.


Let $\epsilon_{i_1, i_2, ..., i_n}$ be the copy of $1 \otimes 1$ for which we assign homological degree $i_j$ in $x_j$ for $j \in \{1,2,...,n\}$.  Iterating the isomorphism of \cite[Lemma 4.3]{BO} and changing notation to better track homological degree, we get an isomorphism of graded $(\Lambda_{\q})^e$-modules $$(\mathbb{K})_m \cong \bigoplus_{i_1+i_2+...+i_n=m} \Lambda_{\textbf{q}} \epsilon_{i_1,i_2,...,i_n} \Lambda_{\textbf{q}},$$ by sending $$x_1^{\alpha_1} \otimes x_1^{\alpha'_1} \otimes^{t_1} x_2^{\alpha_2} \otimes x_2^{\alpha'_2} \otimes^{t_2} ... \otimes^{t_{n-1}} x_n^{\alpha_n} \otimes x_n^{\alpha'_n} \in (\mathbb{K}^{x_1})_{i_1} \otimes^{t_1} (\mathbb{K}^{x_2})_{i_2} \otimes^{t_2} ...\otimes^{t_{n-1}} (\mathbb{K}^{x_n})_{i_n}$$ to $$\prod_{l < j} (-q_{l,j})^{\alpha_j \alpha'_l+i_j \alpha'_l + i_l \alpha_j} x_1^{\alpha_1}x_2^{\alpha_2}...x_n^{\alpha_n} \epsilon_{i_1,i_2,...i_n} x_1^{\alpha'_1} x_2^{\alpha'_2} ... x_n^{\alpha_n'} \in \Lambda_{\textbf{q}} \epsilon_{i_1,i_2,...,i_n} \Lambda_{\textbf{q}}.$$   The boundary map becomes \begin{align*} \delta_m(\epsilon_{i_1,i_2,...,i_n})=&\sum_{j=1}^n ( \prod_{l<j} q_{l,j}^{i_l} x_j \epsilon_{i_1, ...i_{j-1},  i_j-1,i_{j+1}, ..., i_n} \\
& \hspace{10 mm}+ (-1)^{\sum_{l \leq j} i_l} \prod_{l >j} (-q_{j,l})^{i_l} \epsilon_{i_1, ...i_{j-1},  i_j-1,i_{j+1}, ..., i_n} x_j ).
\end{align*}  By an abuse of notation, we will refer to this identified complex as $\K$ with boundary maps $\delta_m$ for the remainder of the paper.


Now that we have a resolution, apply $\Hom_{(\Lambda_{\q})^e}( -, \Lambda_{\q} \rtimes G)$.  Homomorphisms $\eta \in \textrm{Hom}_{(\Lambda_{\q})^e}((\mathbb{K})_m, \Lambda_{\q} \rtimes G)$, are entirely determined by the image, $\eta(\epsilon_{i_1, i_2, ..., i_n})$, for $i_1+i_2+... + i_n = m$.  Denote $\epsilon_{i_1, i_2, ..., i_n}^*$ the function in $\textrm{Hom}_{(\Lambda_{\q})^e}(\mathbb{K}, \Lambda_{\q} \rtimes G)$ defined by $\epsilon_{i_1, i_2, ..., i_n}^*(\epsilon_{j_1, j_2, ..., j_n})=\delta_{i_1, j_1} \delta_{i_2, j_2} ... \delta_{i_n, j_n} \otimes 1$ where $\delta_{i_m, j_m}$ is the Kronecker delta.  Then any $\eta \in \textrm{Hom}_{(\Lambda_{\q})^e}((\mathbb{K})_m, \Lambda_{\q} \rtimes G)$ can be written $$\eta = \sum_{g \in G} \sum_{i_1+i_2+...+i_n=m} (\lambda_{i_1, i_2, ..., i_n}^g \otimes g)\epsilon_{i_1, i_2, ..., i_n}^*$$ where $\lambda_{i_1, i_2, ..., i_n}^g \in \Lambda_{\q}$ depends on $i_1, i_2, ..., i_n$ and $g$.  Moreover, $$\textrm{Hom}_{(\Lambda_{\q})^e}(\mathbb{K}, \Lambda_{\q} \rtimes G) \cong \bigoplus_{g \in G} \textrm{Hom}_{(\Lambda_{\q})^e}(\mathbb{K}, \Lambda_{\q} \otimes g)$$ therefore we can restrict to only looking at homomorphisms in $\textrm{Hom}_{(\Lambda_{\q})^e}(\mathbb{K}, \Lambda_{\q} \otimes g)$ for each $g \in G$.

Before we move farther, we will need some additional notational conveniences, as in \cite{W}.  Elements of $\Lambda_{\q}$ are linear combinations of the monomials $ x_1^{\alpha_1}x_2^{\alpha_2}...x_n^{\alpha_n}$ for $\alpha_i \in \{0,1\}$.  Thus, for $\alpha \in \{0,1\}^n$, denote $x^{\alpha} = x_1^{\alpha_1} x_2^{\alpha_2}...x_n^{\alpha_n}$.  And similarly, for $\beta \in \mathbb{N}^n$, define $\epsilon_{\beta}= \epsilon_{\beta_1, \beta_2, ..., \beta_n}$ and $\epsilon_{\beta}^*= \epsilon_{\beta_1, \beta_2, ..., \beta_n}^*$.  Denote $|\beta|=\beta_1+\beta_2+...+\beta_n$.  Lastly, as in the twisting maps in the construction of $\Lambda_{\q}$, denote $[i]=(0,...,0,1,0,...,0)$, the $n$-tuple with $1$ in the $i$th coordinate and $0$ otherwise.

Using this notation, the induced boundary map on $\Hom_{(\Lambda_{\q})^e}(\mathbb{K}, \Lambda_{\q} \otimes g)$ becomes 
\begin{align*} \delta^m((x^{\alpha} \otimes g) \epsilon_{\beta}^*) =& \sum_{l=1}^n ( \prod_{k <l} q_{k,l}^{\beta_k}x_l(x^{\alpha} \otimes g) \epsilon_{\beta+[l]}^* \\
& \hspace{10 mm}- (-1)^{\sum_{k \leq l} \beta_k} \prod_{k >l} (-q_{l,k})^{\beta_k} (x^{\alpha} \otimes g)x_l \epsilon_{\beta+[l]}^*) \\
=&  \sum_{l=1}^n ( \prod_{k <l} q_{k,l}^{\beta_k}(-q_{k,l})^{-\alpha_k}(x^{\alpha+[l]} \otimes g) \epsilon_{\beta+[l]}^* \\
&\hspace{10 mm} - (-1)^{\sum_{k \leq l} \beta_k} \prod_{k >l} (-q_{l,k})^{\beta_k}(-q_{l,k})^{-\alpha_k} \chi_{g,l}(x^{\alpha+[l]} \otimes g) \epsilon_{\beta+[l]}^*) \\
=&  \sum_{l=1}^n ( \prod_{k <l} (-1)^{\beta_k}(-q_{k,l})^{\beta_k-\alpha_k} - (-1)^{\sum_{k \leq l} \beta_k} \prod_{k >l} (-q_{l,k})^{\beta_k-\alpha_k}\chi_{g,l})\\
& \hspace{20 mm}(x^{\alpha+[l]} \otimes g) \epsilon_{\beta+[l]}^* \\
=&  \sum_{l=1}^n (-1)^{\sum_{k < l} \beta_k} ( \prod_{k <l} (-q_{k,l})^{\beta_k-\alpha_k} - (-1)^{\beta_l} \prod_{k >l} (-q_{l,k})^{\beta_k-\alpha_k}\chi_{g,l})\\
& \hspace{20 mm} (x^{\alpha+[l]} \otimes g) \epsilon_{\beta+[l]}^* \\
\end{align*} for $\beta \in \mathbb{N}^n$ with $|\beta|=m-1$ and $\alpha \in \{0,1\}^n$.  If $\alpha_l=1$ or $(-1)^{\beta_l}\prod_{k \neq l} (-q_{k,l})^{\beta_k-\alpha_k} = \chi_{g,l}$, the coefficient of $\epsilon_{\beta +[l]}^*$ in  $\delta^m((x^{\alpha} \otimes g) \epsilon_{\beta}^*)$ is 0.  Thus, for $l \in \{1,2, ... ,n\}$, define $$\Omega_g(\alpha, \beta, l) = \begin{cases} 0 & \textrm{ if }\alpha_l=1 \\
0 & \textrm{ if }(-1)^{\beta_l}\prod_{k \neq l} (-q_{k,l})^{\beta_k-\alpha_k} = \chi_{g,l} \\
 (-1)^{\sum_{k < l} \beta_k} ( \prod_{k <l} (-q_{k,l})^{\beta_k-\alpha_k} &\\
\hspace{10 mm}- (-1)^{\beta_l} \prod_{k >l} (-q_{l,k})^{\alpha_k-\beta_k}\chi_{g,l}) & \textrm{ otherwise.} \\
\end{cases}$$ 
In this notation, for $\beta \in \mathbb{N}^n$ with $|\beta|=m-1$ and $\alpha \in \{0,1\}^n$, \begin{align} \label{delta} \delta^m((x^{\alpha} \otimes g) \epsilon_{\beta}^*) = \sum_{l=1}^n \Omega_g(\alpha, \beta, l)(x^{\alpha+[l]} \otimes g)\epsilon_{\beta+[l]}^*.\end{align}


\subsection{The vector space} \label{vs}

To compute Hochschild cohomology, we decompose the complex $\Hom_{(\Lambda_{\q})^e}(\K, \Lambda_{\q} \otimes g)$ into subcomplexes, adapting a method of \cite[Section 6]{W} and \cite{NSW}.  For each $m \in \N$, $g \in G$, and $\gamma \in (\N \cup \{-1\})^n$, consider the set, $$K_{g, \gamma}^m = span_k \{(x^{\alpha} \otimes g)\epsilon_{\beta}^* | \alpha \in \{0,1\}^n, \beta \in \mathbb{N}^n, |\beta|=m, \textrm{ and } \beta-\alpha=\gamma\},$$ of all elements in homological degree $m$ with fixed $\beta-\alpha$.    Because the difference $\beta - \alpha$ is unchanged by the boundary map $\delta_m$ given in (\ref{delta}), $K_{g,\gamma}=\bigoplus_{m \in \mathbb{N}} K_{g, \gamma}^m$ is a subcomplex of $\Hom_{(\Lambda)^e}(\mathbb{K}, \Lambda \otimes g)$.  

We will show that for some $g \in G$ and $\gamma \in (\N \cup \{-1\})^n$, the subcomplexes $K_{g, \gamma}$ are acyclic and, for others, the differentials are 0.  The condition of the subcomplexes is determined by the set, $$C_g = \{\gamma \in (\mathbb{N} \cup \{-1\})^n | \forall~ l, \gamma_l=-1 \textrm{ or } (-1)^{\gamma_l}\prod_{k \neq l} (-q_{k,l})^{\gamma_k} = \chi_{g,l} \},$$ of $\gamma$ which satisfy a relation between the quantum coefficients and the diagonal action by $g$.

\begin{lemma} \label{acyclic}
Let $\gamma \in (\N \cup \{-1\})^n - C_g$ for some $g \in G$.  Then $K_{g, \gamma}$ is acyclic.
\end{lemma}

\begin{proof} 
To show this result, we adapt the proof of \cite[Theorem 6.1]{W} and \cite[Lemma 4.6]{NSW} to this setting.  That is, we will show this result by demonstrating a contracting homotopy but first we will need a bit more notation.  For $\gamma \in (\mathbb{N} \cup \{-1\})^n$ and $g \in G$, define $$||\gamma||_g = \# \{l \in \{1, 2, ..., n\} |  \gamma_l \neq -1 \textrm{ and } (-1)^{\gamma_l}\prod_{k \neq l} (-q_{k,l})^{\gamma_k} \neq \chi_{g,l} \}.$$  Notice for $\gamma \in (\mathbb{N} \cup \{-1\})^n-C_g$, $||\gamma||_g \neq 0$ by construction.  

Now, let $m$ and $\gamma \in (\mathbb{N} \cup \{-1\})^n - C_g$ be fixed, and define $h_m: K_{g, \gamma}^m \rightarrow K_{g, \gamma}^{m-1}$ by $$h_m((x^{\alpha} \otimes g) \epsilon_{\beta}^*)= \frac{1}{||\beta-\alpha||_g} \sum_{l=1}^n \omega_g(\alpha, \beta, l)(x^{\alpha-[l]} \otimes g) \epsilon_{\beta-[l]}$$ for $\beta \in \mathbb{N}^n$ with $|\beta|=m$, $\alpha \in \{0,1\}^n$ such that $\beta-\alpha=\gamma$ and $$\omega_g(\alpha, \beta, l) = \begin{cases} 0 & \textrm{ if } \alpha_l=0 \textrm{ or } \beta_l=0 \\
0 & \textrm{ if } \prod_{k \neq l} (-q_{k,l})^{\beta_k-\alpha_k}(-1)^{\beta_l} = -\chi_{g,l} \\
\Omega_g(\alpha-[l],\beta-[l],l)^{-1} & \textrm{ otherwise.} \\ \end{cases}.$$  By defining this contracting homotopy in this way, we require the $\ch k = 0$ condition.

We need to show that $h_{m+1}\delta^{m+1} + \delta^m h_m = \mathds{1} \big|_{K_{g,\gamma}^m}$.  Let $\alpha \in \{0,1\}^n$, $\beta \in \mathbb{N}^n$ with $|\beta|=m$ such that $\beta-\alpha=\gamma$, then \begin{align*} (h_{m+1}\delta^{m+1} +& \delta^m h_m) ((x^{\alpha} \otimes g) \epsilon_{\beta}^*) \\
=& \frac{1}{||\beta-\alpha||_g} \sum_{j=1}^n \sum_{l=1}^n (\Omega_g(\alpha, \beta, l)\omega_g(\alpha+[l], \beta+[l], j) \\
&\hspace{25 mm}+ \Omega_g(\alpha-[j], \beta-[j], l) \omega_g(\alpha, \beta, j))(x^{\alpha+[l]-[j]} \otimes g) \epsilon_{\beta+[l]-[j]}^*\\
=& \frac{1}{||\beta-\alpha||_g} \sum_{l=1}^n (\Omega_g(\alpha, \beta, l)\omega_g(\alpha+[l], \beta+[l], l) \\
&\hspace{25 mm}+\Omega_g(\alpha-[l], \beta-[l], l) \omega_g(\alpha, \beta, l))(x^{\alpha} \otimes g) \epsilon_{\beta}^* \\
&+\frac{1}{||\beta-\alpha||_g} \sum_{j \neq l} (\Omega_g(\alpha, \beta, l)\omega_g(\alpha+[l], \beta+[l], j)\\
&\hspace{25 mm}+\Omega_g(\alpha-[j], \beta-[j], l) \omega_g(\alpha, \beta, j))(x^{\alpha+[l]-[j]} \otimes g) \epsilon_{\beta+[l]-[j]}^*.
\end{align*}

It remains to show that 
\begin{align*}
\frac{1}{||\beta-\alpha||_g} \sum_{l=1}^n (\Omega_g(\alpha, \beta, l)\omega_g(\alpha+[l], \beta+[l], l) &+\Omega_g(\alpha-[l], \beta-[l], l) \omega_g(\alpha, \beta, l))(x^{\alpha} \otimes g) \epsilon_{\beta}^* \\
& = (x^{\alpha} \otimes g) \epsilon_{\beta}^* 
\end{align*} and \begin{align*} 
\frac{1}{||\beta-\alpha||_g} \sum_{j \neq l} (\Omega_g(\alpha, \beta, l)&\omega_g(\alpha+[l], \beta+[l], j) \\
&+\Omega_g(\alpha-[j], \beta-[j], l) \omega_g(\alpha, \beta, j))(x^{\alpha+[l]-[j]} \otimes g) \epsilon_{\beta+[l]-[j]}^* \\
& \hspace{36 mm} = 0. 
\end{align*}  

Let's start by showing the second condition.  Assume $j \neq l$.  Notice, because $(\alpha \pm [r])_s=\alpha_s$ and $(\beta \pm [r])_s = \beta_s$ for $r \neq s$, $\Omega_g(\alpha, \beta, l)\omega_g(\alpha+[l], \beta+[l], j)$ and $\Omega_g(\alpha-[j], \beta-[j], l) \omega_g(\alpha, \beta, j)$ are both either simultaneously $0$ or nonzero.  If they are $0$, we have nothing to prove.  Therefore assume these terms are nonzero.  If $j<l$, then 
\begin{align*}\Omega_g(\alpha, \beta, l)\omega_g(\alpha+[l]&, \beta+[l], j)+\Omega_g(\alpha-[j], \beta-[j], l) \omega_g(\alpha, \beta, j) \\
=& (-1)^{\sum_{k<l} \beta_k}(\prod_{k<l} (-q_{k,l})^{\beta_k-\alpha_k}-(-1)^{\beta_l}\prod_{k>l}(-q_{l,k})^{-\beta_k+\alpha_k} \chi_{g,l}) \\
&\hspace{10 mm} (-1)^{\sum_{k<j} \beta_k}(\prod_{k<j}(-q_{k,j})^{\beta_k-\alpha_k} + (-1)^{\beta_j} \prod_{k>j} (-q_{j,k})^{-\beta_k+\alpha_k} \chi_{g,j})^{-1}\\
& + (-1)^{\sum_{k<l} \beta_k}(-\prod_{k <l} (-q_{k,l})^{\beta_k-\alpha_k}+(-1)^{\beta_l}\prod_{k>l}(-q_{l,k})^{-\beta_k+\alpha_k} \chi_{g,l})\\
&\hspace{10 mm} (-1)^{\sum_{k<j} \beta_k}(\prod_{k<j}(-q_{k,j})^{\beta_k-\alpha_k} + (-1)^{\beta_j} \prod_{k>j} (-q_{j,k})^{-\beta_k+\alpha_k} \chi_{g,j})^{-1}\\
=& 0.
\end{align*}
We get a similar computation if $j>l$.  Thus for every $j, l \in \{1, 2, ... ,n\}$ with $j \neq l$, $$\Omega_g(\alpha, \beta, l)\omega_g(\alpha+[l], \beta+[l], j)+\Omega_g(\alpha-[j], \beta-[j], l) \omega_g(\alpha, \beta, j)=0,$$ giving us our desired expression.

Now we will show \begin{align*}
\frac{1}{||\beta-\alpha||_g} \sum_{l=1}^n (\Omega_g(\alpha, \beta, l)\omega_g(\alpha+[l], \beta+[l], l) &+\Omega_g(\alpha-[l], \beta-[l], l) \omega_g(\alpha, \beta, l))(x^{\alpha} \otimes g) \epsilon_{\beta}^* \\
& = (x^{\alpha} \otimes g) \epsilon_{\beta}^* 
\end{align*} by showing that $$\sum_{l=1}^n (\Omega_g(\alpha, \beta, l)\omega_g(\alpha+[l], \beta+[l], l) +\Omega_g(\alpha-[l], \beta-[l], l) \omega_g(\alpha, \beta, l))=||\beta-\alpha||_g.$$  

Recall, we assumed $\gamma \in (\mathbb{N} \cup \{-1\})^n - C_g$.  Thus for some $l \in \{1,2,...,n\}$, $\beta-\alpha=\gamma \in (\mathbb{N} \cup \{-1\})^n - C_g$ has $\beta_l-\alpha_l \neq -1$ and $(-1)^{\beta_l-\alpha_l}\prod_{k \neq l} (-q_{k,l})^{\beta_k-\alpha_k} \neq \chi_{g,l}$.  The reader can check that $\Omega_g(\alpha, \beta, l)\omega_g(\alpha+[l], \beta+[l], l) +\Omega_g(\alpha-[l], \beta-[l], l) \omega_g(\alpha, \beta, l)=1$ for each $l \in \{1,2,...,n\}$ such that $\beta_l-\alpha_l \neq -1$ and $(-1)^{\beta_l-\alpha_l}\prod_{k \neq l} (-q_{k,l})^{\beta_k-\alpha_k} \neq \chi_{g,l}$.  On the other hand, if $l \in \{1,2,...,n\}$ satisfies $\beta_l-\alpha_l = -1$ or \newline $(-1)^{\beta_l-\alpha_l}\prod_{k \neq l} (-q_{k,l})^{\beta_k-\alpha_k} = \chi_{g,l}$, then $\Omega_g(\alpha, \beta, l)\omega_g(\alpha+[l], \beta+[l], l) +\Omega_g(\alpha-[l], \beta-[l], l) \omega_g(\alpha, \beta, l)=0$.  Combined these give us our final result.

\end{proof}

\begin{thm} \label{vsthm}  For each $g$ in $G$, 
$$\HH^m(\Lambda_{\q}, \Lambda_{\q} \otimes g) \cong \bigoplus_{\substack{\beta \in \mathbb{N}^n \\ |\beta|=m}} \bigoplus_{\substack{\alpha \in \{0,1\}^n \\ \beta-\alpha \in C_g}} span_k \{(x^{\alpha} \otimes g) \epsilon_{\beta}^*\}.$$  Therefore $\HH^m(\Lambda \rtimes G)$ is the $G$-invariant subspace of $$\bigoplus_{g \in G} \bigoplus_{\substack{\beta \in \mathbb{N}^n \\ |\beta|=m}} \bigoplus_{\substack{\alpha \in \{0,1\}^n \\ \beta-\alpha \in C_g}} span_k \{(x^{\alpha} \otimes g) \epsilon_{\beta}^*\}.$$
\end{thm} 

\begin{proof}

We start by showing that for $\gamma \in C_g$, $\delta^m \big|_{K_{g,\gamma}^m} = 0$.  If $\gamma \in C_g$, then for all $l \in \{1,2,...,n\}$, $\gamma_l=-1$ or $(-1)^{\gamma_l}\prod_{k \neq l} (-q_{k,l})^{\gamma_k} = \chi_{g,l}$. 

\textbf{Case 1}: If $\gamma_l=-1$, then, because $\gamma_l=-1$ if and only if $\alpha_l=1$ and $\beta_l=0$, we have $\Omega_g(\alpha, \beta, l)=0$.  

\textbf{Case 2}: If $ (-1)^{\gamma_l}\prod_{k \neq l} (-q_{k,l})^{\gamma_k} = \chi_{g,l}$, we have two cases to consider, $\alpha_l = 0$ or $\alpha_l=1$.  If additionally $\alpha_l=1$, then $\Omega_g(\alpha, \beta, l)=0$.  In this case, the $ (-1)^{\gamma_l}\prod_{k \neq l} (-q_{k,l})^{\gamma_k} = \chi_{g,l}$ condition is unnecessary.  If $ (-1)^{\gamma_l}\prod_{k \neq l} (-q_{k,l})^{\gamma_k} = \chi_{g,l}$ and $\alpha_l=0$, then $ (-1)^{\beta_l}\prod_{k \neq l} (-q_{k,l})^{\gamma_k} =  (-1)^{\gamma_l}\prod_{k \neq l} (-q_{k,l})^{\gamma_k} = \chi_{g,l}$ and therefore $\Omega_g(\alpha, \beta, l)=0$ in this case as well.  

Thus for $\gamma \in C_g$, $\delta^m \big|_{K_{g,\gamma}^m} = 0$, making the cohomology of the subcomplexes $K_{g,\gamma}^m$ for $\gamma \in C_g$ equal to $span_k\{(x^{\alpha} \otimes g)\epsilon_{\beta}^* | \beta-\alpha=\gamma\}$.  By Lemma \ref{acyclic}, the subcomplexes $K_{g, \gamma}^m$ for $\gamma \notin C_g$ are acyclic.

\end{proof}


\subsection{The cup product} \label{cup}



We define the cup product on Hochschild cohomology as a composition, for $f \in \Hom_{(\Lambda_q)^e}(\mathbb{B}(\Lambda_{\q})_l, \Lambda_{\q} \rtimes G)$ and $g \in \Hom_{(\Lambda)^e}(\mathbb{B}(\Lambda_{\q})_m, \Lambda_{\q} \rtimes G)$ $$f \smile g: \mathbb{B}(\Lambda_{\q}) \xrightarrow{\Delta_{\mathbb{B}(\Lambda_{\q})}} \mathbb{B}(\Lambda_{\q}) \otimes_{\Lambda_{\q}} \mathbb{B}(\Lambda_{\q}) \xrightarrow{f \otimes_{\Lambda_{\q}} g} \Lambda_{\q} \rtimes G \otimes_{\Lambda_{\q}} \Lambda_{\q} \rtimes G \xrightarrow{\mu} \Lambda_{\q} \rtimes G$$ where $\mu$ is the multiplication map.  To make use of this description, we will define yet another resolution of $\Lambda_{\q}$,  $\mathbb{P}$, which is a subcomplex of $\mathbb{B}(\Lambda)$ such that the diagonal $\Delta_{\mathbb{B}(\Lambda)}$ induces a comultiplication on $\mathbb{P}$.


We begin by defining an $n-$dimensional analog of the $f_i^m$ defined in \cite{BGMS}.  Let $f_{(0,0,...,0,0)} =1, f_{[l]}=x_l$ for all $l \in \{1,2,...,n\}$, and $f_{\beta}=0$ for any $\beta \in \Z^n$ with $\beta_l < 0$ for some $l \in \{1,2,..., n\}$.  Then for $\beta \in \N^n$, define $$f_{\beta} = \sum_{l=1}^n \prod_{k>l} q_{l,k}^{\beta_k} f_{\beta-[l]} \otimes x_l.$$

That is, for $\beta \in \mathbb{N}^n$, $f_{\beta}$ is a linear combination of all tensor products of length $|\beta|$ with $\beta_i$ $x_i$'s for each $i \in \{1,2, ...,n\}$ and coefficient $\mathbf{q}^\alpha$ determined by the commuting coefficients that appear when moving the generators past each other starting from the configuration with generators in increasing order.  For example, 
$$f_{(0,2,1,0,...,0)} = x_2 \otimes x_2 \otimes x_3 + q_{2,3} x_2 \otimes x_3 \otimes x_2 + q_{2,3}^2 x_3 \otimes x_2 \otimes x_2.$$


As in \cite{BGMS}, $\tilde{f_{\beta}}=1 \otimes f_{\beta} \otimes 1$.  Consider the $\Lambda_{\q}^e$-module resolution $$\mathbb{P}: ... \xrightarrow{d_{\mathbb{P}}^{m+1}} \bigoplus_{\substack{\beta \in \mathbb{N}^n \\ |\beta|=m}} \Lambda_{\mathbf{q}} \otimes f_{\beta} \otimes \Lambda_{\mathbf{q}} \xrightarrow{d_{\mathbb{P}}^{m}} \bigoplus_{\substack{\beta \in \mathbb{N}^n \\ |\beta|=m-1}} \Lambda_{\mathbf{q}} \otimes f_{\beta} \otimes \Lambda_{\mathbf{q}} \xrightarrow{d_{\mathbb{P}}^{m-1}} ...$$  where, for $\beta \in \N^n$ with $|\beta|=m$, $$d_{\mathbb{P}}(\tilde{f_{\beta}})= \sum_{j=1}^n ( \prod_{l<j} q_{l,j}^{i_l} x_j \tilde{f}_{\beta-[j]} + (-1)^m \prod_{l >j} q_{j,l}^{i_l} \tilde{f}_{\beta-[j]} x_j ).$$  

\begin{lemma}
$\mathbb{P}$ is a subcomplex of $\mathbb{B}$.
\end{lemma}

\begin{proof}
To see this, we will show that the differential on the bar resolution, $d$, induces the map $d_{\mathbb{P}}$ defined above.  Notice for each $m \in \mathbb{N}$, $(\mathbb{P})_m \subset (\mathbb{B}(\Lambda))_m$.  


Fix $\beta \in \mathbb{N}$ with $|\beta|=m$.  Write $$d=\sum_{i=0}^{m} (-1)^i\delta_i$$ where $\delta_i(\lambda_0 \otimes \lambda_1 \otimes ... \otimes \lambda_{m+1})=\lambda_0 \otimes \lambda_1 \otimes ... \otimes \lambda_i \lambda_{i+1} \otimes ... \otimes \lambda_{m+1}$ for $\lambda_0, \lambda_1, ..., \lambda_{m+1} \in \Lambda_{\q}$.  Using this notation,  $d(\tilde{f}_{\beta})$ contains the terms $$\sum_{j=1}^n ( \prod_{l<j} q_{l,j}^{i_l} x_j \tilde{f}_{\beta-[j]} + (-1)^m \prod_{l >j} q_{j,l}^{i_l} \tilde{f}_{\beta-[j]} x_j )$$ as they are generated by $\delta_0(\tilde{f}_{\beta})$ and $\delta_m(\tilde{f}_{\beta})$ respectively.  Thus what needs to be shown is that $\sum_{i=1}^{m-1} (-1)^i\delta_i(\tilde{f}_{\beta})=0$.


For $\alpha \in \{1, 2, ..., n\}^m$ define $x_{\alpha} = x_{\alpha_1} \otimes x_{\alpha_2} \otimes ... \otimes x_{\alpha_m}$.  Then we can write $$f_{\beta}=\sum_{\substack{\alpha \in \{1, 2, ..., n\}^m \\ \#\{\alpha_l=i\}=\beta_i ~\forall i \in \{1,2,...,n\}}} \mathbf{q}^{\alpha}x_{\alpha}$$ where $\mathbf{q}^{\alpha}$ is determined by the commuting coefficients that appear when moving the generators past each other starting from the configuration with generators in increasing order.  
 
Consider a single term $\mathbf{q}^{\alpha}\otimes x_{\alpha}\otimes 1$ in $\tilde{f}_{\beta}$.  If $\alpha_i = \alpha_{i+1}$ for $i \in \{2,3, ..., m-2\}$ then, because $x_j^2 = 0$ for all $j \in \{1,2,...,n\}$, $\delta_i(\mathbf{q}^{\alpha}\otimes x_{\alpha}\otimes 1)=0$.  If $\alpha_i \neq \alpha_{i+1}$ for $i \in \{2,3, ..., m-2\}$ then, without loss of generality, assume $\alpha_i < \alpha_{i+1}$.  By definition $\tilde{f}_{\beta}$ also contains the term $\mathbf{q}^{\alpha}q_{\alpha_i, \alpha_{i+1}}\otimes x_{\alpha_1} \otimes x_{\alpha_2} \otimes ... \otimes x_{\alpha_{i-1}} \otimes x_{\alpha_{i+1}} \otimes x_{\alpha_i} \otimes x_{\alpha_{i+2}} \otimes ... \otimes x_{\alpha_m} \otimes 1$.  Notice \begin{align*} \delta_i&(\mathbf{q}^{\alpha}\otimes x_{\alpha}\otimes 1) \\
&=\mathbf{q}^{\alpha}\otimes x_{\alpha_1} \otimes x_{\alpha_2} \otimes ... \otimes x_{\alpha_{i-1}} \otimes  x_{\alpha_{i}} x_{\alpha_{i+1}} \otimes x_{\alpha_{i+2}} \otimes ... \otimes x_{\alpha_m} \otimes 1 \\
&= \mathbf{q}^{\alpha}(-q_{\alpha_i, \alpha_{i+1}})\otimes x_{\alpha_1} \otimes x_{\alpha_2} \otimes ... \otimes x_{\alpha_{i-1}} \otimes x_{\alpha_{i+1}} x_{\alpha_i} \otimes x_{\alpha_{i+2}} \otimes ... \otimes x_{\alpha_m} \otimes 1\\
&= - \mathbf{q}^{\alpha}q_{\alpha_i, \alpha_{i+1}}\otimes x_{\alpha_1} \otimes x_{\alpha_2} \otimes ... \otimes x_{\alpha_{i-1}} \otimes x_{\alpha_{i+1}} x_{\alpha_i} \otimes x_{\alpha_{i+2}} \otimes ... \otimes x_{\alpha_m} \otimes 1 \end{align*}  and \begin{align*} \delta_i&(\mathbf{q}^{\alpha}q_{\alpha_i, \alpha_{i+1}}\otimes x_{\alpha_1} \otimes x_{\alpha_2} \otimes ... \otimes x_{\alpha_{i-1}} \otimes x_{\alpha_{i+1}} \otimes x_{\alpha_i} \otimes x_{\alpha_{i+2}} \otimes ... \otimes x_{\alpha_m} \otimes 1) \\
&=\q^{\alpha}q_{\alpha_i, \alpha_{i+1}}\otimes x_{\alpha_1} \otimes x_{\alpha_2} \otimes ... \otimes x_{\alpha_{i-1}} \otimes x_{\alpha_{i+1}} x_{\alpha_i} \otimes x_{\alpha_{i+2}} \otimes ... \otimes x_{\alpha_m} \otimes 1.\end{align*}  Thus in $d(\tilde{f}_{\beta})$, these two terms cancel each other and no other terms in $\tilde{f}_{\beta}$ contribute to the $1 \otimes x_{\alpha_1} \otimes x_{\alpha_2} \otimes ... \otimes x_{\alpha_{i-1}} \otimes x_{\alpha_{i+1}} x_{\alpha_i} \otimes x_{\alpha_{i+2}} \otimes ... \otimes x_{\alpha_m} \otimes 1$ term in $d(\tilde{f}_{\beta}).$  
\end{proof}


To see that the diagonal on $\mathbb{B}(\Lambda_{\q})$ restricts to a diagonal on $\mathbb{P}$, recall $$\Delta_{\mathbb{B}}(\lambda_0 \otimes ... \otimes \lambda_{m+1}) = \sum_{i=0}^m (\lambda_0 \otimes ... \otimes \lambda_i \otimes 1) \otimes_{\Lambda_{\q}} (1 \otimes \lambda_{i+1} \otimes ... \otimes \lambda_{m+1})$$ for $\lambda_0, ..., \lambda_{m+1} \in \Lambda_{\q}$.  For $\beta \in \mathbb{N}^m$ and $0 \leq t \leq |\beta|$ fixed, we have $$f_{\beta}= \sum_{\substack{ \alpha+ \gamma =\beta \\ \alpha, \gamma \in \N^m \textrm{ and } |\alpha|=t}}(\prod_{\substack{1 \leq l \leq n \\ k<l}} q_{k,l}^{\gamma_k \alpha_l}) f_{\alpha} \otimes f_{\gamma}.$$  Therefore define, $$\Delta_{\mathbb{P}}(\tilde{f}_{\beta})=\sum_{\alpha+\gamma=\beta} (\prod_{\substack{1 \leq l \leq n \\ k<l}} q_{k,l}^{\gamma_k \alpha_l}) \tilde{f}_{\alpha} \otimes_{\Lambda_{\q}} \tilde{f}_{\gamma}.$$  Then the diagonal map, $\Delta_{\mathbb{P}}$, is induced by $\Delta_{\mathbb{B}}$.



Finally, the motivation for developing $f_{\beta}$ is that we have $\bigoplus_{\substack{\beta \in \mathbb{N}^n \\ |\beta|=m}} \Lambda_{\mathbf{q}}^n \otimes f_{\beta} \otimes \Lambda_{\mathbf{q}}^n \cong \bigoplus_{\substack{\beta \in \mathbb{N}^n \\ |\beta|=m}} \Lambda_{\q} \epsilon_{\beta} \Lambda_{\q}$ by sending $\tilde{f}_{\beta}$ to $\epsilon_{\beta}$ and this isomorphism preserves the differential on the respective complexes.  That is, $\K \cong \mathbb{P}$.  Therefore, using this isomorphism, we have our desired diagonal map $$\Delta_{\mathbb{K}}(\epsilon_{\beta})=\sum_{\alpha+\gamma = \beta} \prod_{\substack{1 \leq l \leq n \\ k<l}} q_{k,l}^{\gamma_k \alpha_l} \epsilon_{\alpha} \otimes_{\Lambda_{\mathbf{q}}^n} \epsilon_{\gamma}$$ and we are finally ready to describe the cup product.

By Theorem \ref{vsthm}, the vector space structure of $\HH^m(\Lambda, \Lambda \rtimes G)$ has a basis given by $(x^{\alpha} \otimes g) \epsilon_{\beta}^*$ such that $\beta-\alpha \in C_g$.  Then, the cup product, defined on these basis elements is given by the following formula.

\begin{thm} \label{cupthm}
If $\alpha, \gamma \in \{0,1\}^n$, $\beta, \kappa \in \mathbb{N}^n$, and $g, h \in G$, then $$(x^{\alpha} \otimes g) \epsilon_{\beta}^* \smile (x^{\gamma} \otimes h) \epsilon_{\kappa}^* = \prod_{l=1}^n \chi_{g,l}^{\gamma_l} \prod_{k<l} q_{k,l}^{\kappa_k \beta_l-\gamma_k \alpha_l}(-1)^{-\gamma_k \alpha_l} (x^{\alpha+\gamma} \otimes gh) \epsilon_{\beta +\kappa}^*.$$
\end{thm}

\begin{proof}
Let $\alpha, \gamma \in \{0,1\}^n$, $\beta, \kappa, \rho \in \mathbb{N}^n$, and $g, h \in G$.  Because we can identify $\K$ as a subcomplex of $\mathbb{B}(\Lambda)$, we can compute the cup product, $(x^{\alpha} \otimes g) \epsilon_{\beta}^* \smile (x^{\gamma} \otimes h) \epsilon_{\kappa}^*$ as the composition 
$$\mathbb{K} \xrightarrow{\Delta_{\mathbb{K}}} \mathbb{K} \otimes_{\Lambda_{\q}} \mathbb{K} \xrightarrow{(x^{\alpha} \otimes g) \epsilon_{\beta}^* \otimes (x^{\gamma} \otimes h) \epsilon_{\kappa}^*} \Lambda_{\q} \rtimes G \otimes_{\Lambda_{\q}} \Lambda_{\q} \rtimes G \xrightarrow{\mu} \Lambda_{\q} \rtimes G.$$

Then 
\begin{align*}
(x^{\alpha} \otimes g) \epsilon_{\beta}^* &\smile (x^{\gamma} \otimes h) \epsilon_{\kappa}^* (\epsilon_{\rho}) \\
&= \mu ((x^{\alpha} \otimes g) \epsilon_{\beta}^* \otimes (x^{\gamma} \otimes h) \epsilon_{\kappa}^*(\Delta_{\mathbb{K}}(\epsilon_{\rho}))) \\
&= \mu ((x^{\alpha} \otimes g) \epsilon_{\beta}^* \otimes (x^{\gamma} \otimes h) \epsilon_{\kappa}^*(\sum_{\rho'+\rho'' = \rho} \prod_{\substack{1 \leq l \leq n \\ k<l}} q_{k,l}^{\rho''_k \rho'_l} \epsilon_{\rho'} \otimes_{\Lambda_{\mathbf{q}}^n} \epsilon_{\rho''})) \\
&= \mu ( \prod_{\substack{1 \leq l \leq n \\ k<l}} q_{k,l}^{\kappa_k \beta_l} (x^{\alpha} \otimes g) \otimes_{\Lambda_{\mathbf{q}}^n} (x^{\gamma} \otimes h)) \\
&= \prod_{\substack{1 \leq l \leq n \\ k<l}} q_{k,l}^{\kappa_k \beta_l} (x^{\alpha} \otimes g) (x^{\gamma} \otimes h) \\
&= \prod_{l=1}^n \chi_{g,l}^{\gamma_l} \prod_{k<l} q_{k,l}^{\kappa_k \beta_l} (x^{\alpha} x^{\gamma} \otimes gh) \\
&= \prod_{l=1}^n \chi_{g,l}^{\gamma_l} \prod_{k<l} q_{k,l}^{\kappa_k \beta_l}(-q_{k,l})^{-\gamma_k \alpha_l} (x^{\alpha+\gamma} \otimes gh). \\
\end{align*}

Notice for the third equality, we need $\rho'=\beta, \rho''=\kappa$, and $\rho=\kappa+\beta$.  
\end{proof}


\subsection{The Gerstenhaber bracket} \label{bracket}

With a bit more structure, we can compute the brackets on $\HH^*(\Lambda_{\mathbf{q}}^n \#G)$ using the techniques of \cite{NW} adapted to this setting.  

\begin{lemma}
$\mathbb{K}$ satisfies Conditions \ref{cond}.
\end{lemma}

\begin{proof}
(a) Let $\iota: \mathbb{K} \rightarrow \mathbb{B}$ be defined by sending $\epsilon_{\beta} \mapsto \tilde{f_{\beta}}$, with $\epsilon_{\beta}$, $\tilde{f_{\beta}}$ defined as in the previous section.

(b) Extend $\tilde{f_{\beta}}$ to a free $\Lambda_{\q}^e$-basis of $\B_{|\beta|}$.  Let $\pi: \mathbb{B} \rightarrow \mathbb{K}$ be a map $\pi$ such that $\tilde{f_{\beta}} \mapsto \epsilon_{\beta}$.  By the Comparison Theorem, such a map exists and, by construction, $\pi \iota = \mathds{1}_{\mathbb{K}}$.

(c) Let the diagonal map $\Delta_{\mathbb{K}}: \mathbb{K} \rightarrow \mathbb{K} \otimes_{\Lambda_{\mathbf{q}}^n} \mathbb{K}$ be defined by $$\Delta_{\mathbb{K}}(\epsilon_{\beta})= \sum_{\alpha + \gamma = \beta} \prod_{\substack{1 \leq l \leq n \\ k<l}} q_{k,l}^{\gamma_k \alpha_l} \epsilon_{\alpha} \otimes_{\Lambda_{\mathbf{q}}^n} \epsilon_{\gamma}$$ for $\beta \in \mathbb{N}^n$, as in Section \ref{cup}.  The reader may check that $\Delta_{\mathbb{B}}\iota = (\iota \otimes_{\Lambda_{\mathbf{q}}^n} \iota ) \Delta_{\mathbb{K}}$ as required.

\end{proof}

Let $\phi: \K \otimes_A \K \rightarrow \K$ satisfying $d(\phi) = F_{\K}$.  We can define the $\circ$-product on $\HH^*(\Lambda)$ on the chain level as a composition 
$$f \circ g : \K \xrightarrow{\Delta_{\K}} \K \otimes_{\Lambda} \K \xrightarrow{\Delta_{\K} \otimes_{\Lambda} \mathds{1}_{\K}} \K \otimes_{\Lambda} \K \otimes \K \xrightarrow{\mathds{1}_{\K} \otimes_{\Lambda} g \otimes_{\Lambda} \mathds{1}_{\K}} \K \otimes_{\Lambda} \K \xrightarrow{\phi} \K \xrightarrow{f} {\Lambda}$$
for $f \in \Hom_{{\Lambda}^e}((\K)_l, {\Lambda})$ and $g \in \Hom_{{\Lambda}^e}((\K)_m, {\Lambda})$.  However, we would like to define the $\circ$-product on $\HH^*(\Lambda \rtimes G)$.  By \cite{NW2}, we can define such a $\circ$-product using a similar technique, extended trivially to the group. 

Define $\tilde{\mathbb{K}}$ to be the resolution $$ ...\xrightarrow{\tilde{\delta}_3} \mathbb{K}_2 \otimes kG \xrightarrow{\tilde{\delta}_2} \mathbb{K}_1 \otimes kG \xrightarrow{\tilde{\delta}_1} \mathbb{K}_0 \otimes kG \xrightarrow{\tilde{\delta}_0} \Lambda_{\q} \rtimes G  \rightarrow 0$$  where $\tilde{\delta}_m=\delta_m \otimes \mathds{1}_{kG}$.  Let $\Delta_{\tilde{\mathbb{K}}}=\Delta_{\mathbb{K}}\otimes \mathds{1}_{kG}$ be the induced diagonal map on $\tilde{\K}$.  Let $\tilde{\phi}=\phi \otimes \mathds{1}_{kG}$.  Then, by \cite[Section~2.2]{NW2}, for $f \in \Hom_{(\Lambda_{\q} \rtimes G)^e}(\tilde{\mathbb{K}}_m, \Lambda_{\q} \rtimes G)$ and $g \in \Hom_{(\Lambda_{\q} \rtimes G)^e}(\tilde{\mathbb{K}}_{l}, \Lambda_{\q} \rtimes G)$ we can view the $\phi$-circle product, $f \circ_{\phi} g$, as a composition $$\tilde{\mathbb{K}} \xrightarrow{\Delta_{\tilde{\mathbb{K}}}} \tilde{\K} \otimes_{\Lambda_{\q} \rtimes G} \tilde{\K} \xrightarrow{\mathds{1}_{\tilde{\K}}} \tilde{\mathbb{K}} \otimes_{\Lambda_{\q} \rtimes G} \tilde{\mathbb{K}} \otimes_{\Lambda_{\q} \rtimes G} \tilde{\mathbb{K}} \xrightarrow{\mathds{1}_{\tilde{\mathbb{K}}} \otimes g \otimes \mathds{1}_{\tilde{\mathbb{K}}}}  \tilde{\mathbb{K}} \otimes_{\Lambda_{\q} \rtimes G} \tilde{\mathbb{K}} \xrightarrow{\tilde{\phi}} \tilde{\mathbb{K}} \xrightarrow{f} \Lambda_{\q} \rtimes G$$ where the tensor products in $\mathds{1}_{\tilde{\mathbb{K}}} \otimes g \otimes \mathds{1}_{\tilde{\mathbb{K}}}$ are over ${\Lambda_{\q} \rtimes G}$, $\mathds{1}_{\tilde{\mathbb{K}}} \otimes g \otimes \mathds{1}_{\tilde{\mathbb{K}}}$ includes the identification $\tilde{\K} \otimes_{\Lambda_{\q} \rtimes G} \Lambda_{\q} \rtimes G \cong \tilde{\K}$, and this function has Koszul signs as in (\ref{circdef}).  By \cite[Theorem~3.2.5]{NW}, the Gerstenhaber bracket on $\HHD(\Lambda_{\q} \rtimes G)$ is given by $$[f, g] = f \circ_{\tilde{\phi}} g - (-1)^{(m-1)(l-1)}g \circ_{\tilde{\phi}} f$$ at the chain level.  Thus the only remaining work is to determine $\phi$.

Using \cite[Lemma 3.5]{GNW}, we can define $\phi$ iteratively as $\mathbb{K}=\textrm{Tot}(...((\mathbb{K}^{x_1} \otimes^{t_1} \mathbb{K}^{x_2}) \otimes^{t_2} ...)\otimes^{t_{n-1}} \mathbb{K}^{x_n})$ is defined iteratively.  To get a closed form description of $\phi$, we need to introduce some notation.  Let $\mathbb{K}^{(\ell)} = \textrm{Tot}(...((\mathbb{K}^{x_1} \otimes^{t_1} \mathbb{K}^{x_2}) \otimes^{t_2} ...)\otimes^{t_{\ell-1}} \mathbb{K}^{x_\ell})$ and for $\alpha \in \{0,1\}^{\ell}$, define $\alpha_{(l-1)}=(\alpha_1, \alpha_2, ..., \alpha_{\ell-1})$ the $(\ell-1)$-tuple consisting of the first $\ell-1$ entries of $\alpha$.  

\begin{lemma}
If $\beta, \gamma \in \mathbb{N}^n$ and $\alpha \in \{0,1\}^n$, then 
\begin{align*} \phi(\epsilon_{\beta} \otimes_{\Lambda_{\mathbf{q}}^n} x^{\alpha} \epsilon_{\gamma})=&\sum_{\ell=1}^n (-1)^{|\beta|} \delta_{\beta_{\ell+1}+...+\beta_{n}, 0} \delta_{\gamma_1 + ... + \gamma_{\ell-1}, 0} \delta_{\alpha_{\ell}, 1} \prod_{\ell < k \leq n} (-q_{\ell, k})^{\alpha_k(\gamma_{\ell}+1)} \\ 
& \hspace{5 mm} \prod_{1 < k \leq \ell} (-q_{k, \ell})^{\alpha_k(\beta_{\ell}+1)} \prod_{\substack{1 \leq r < s \leq n \\ r \neq \ell \neq s}} (-q_{r, s})^{\alpha_r(\alpha_s+\gamma_s)+\alpha_s \beta_r} \\
& \hspace{10 mm} x_{\ell + 1}^{\alpha_{\ell+1}}...x_n^{\alpha_n} \epsilon_{\beta+\gamma+[\ell]} x_1^{\alpha_1}...x_{\ell-1}^{\alpha_{\ell-1}}
\end{align*} 
\end{lemma}

\begin{proof}
We will show this by induction.  By \cite[Lemma 4.2]{GNW}, we know for $\beta, \gamma \in \mathbb{N}$ and $\alpha \in \{0,1\}$ $$\phi_{\mathbb{K}^{(1)}}(\epsilon_{\beta} \otimes_{\Lambda_{\mathbf{q}}^n} x_1^{\alpha} \epsilon_{\gamma})=\delta_{\alpha,1}(-1)^{\beta} \epsilon_{\alpha+\gamma+1}.$$

Assume the formula holds for $\phi_{\mathbb{K}^{(n-1)}}$.  Then for $\beta, \gamma \in \mathbb{N}^n$ and $\alpha \in \{0,1\}^n$, using \cite[Lemma 3.5]{GNW} in the second equality, 
\begin{align*}
\phi_{\mathbb{K}^{(n)}}&(\epsilon_{\beta} \otimes_{\Lambda_{\mathbf{q}}^n} x^{\alpha} \epsilon_{\gamma})= \phi (\prod_{k<n} (-q_{k,n})^{\alpha_n \gamma_k} \epsilon_{\beta} \otimes_{\Lambda_{\mathbf{q}}^n} x^{\alpha_{(n-1)}} \epsilon_{\gamma_{(n-1)}} \otimes x^{\alpha_n} \epsilon_{\gamma_n})\\
=& (\phi_{\mathbb{K}^{(n-1)}} \otimes F_{\mathbb{K}^{x_n}}^l + (-1)^{i+p}F_{\mathbb{K}^{(n-1)}}^r \otimes \phi_{\mathbb{K}^{x_n}}) \sigma \\
& \hspace{5 mm} (\prod_{k<n} (-q_{k,n})^{\alpha_n \gamma_k} \epsilon_{\beta} \otimes_{\Lambda_{\mathbf{q}}^n} x^{\alpha_{(n-1)}} \epsilon_{\gamma_{(n-1)}} \otimes x^{\alpha_n} \epsilon_{\gamma_n}) \\
=& (\phi_{\mathbb{K}^{(n-1)}} \otimes F_{\mathbb{K}^{x_n}}^l + (-1)^{i+p}F_{\mathbb{K}^{(n-1)}}^r \otimes \phi_{\mathbb{K}^{x_n}}) \\
& \hspace{5 mm}(\prod_{k<n} (-q_{k,n})^{\alpha_n \gamma_k-\beta_n(\alpha_k + \gamma_k)}(-1)^{\beta_n |\gamma_{(n-1)}|} \epsilon_{\beta_{(n-1)}} \otimes_{\Lambda_{\mathbf{q}}^n} x^{\alpha_{(n-1)}} \epsilon_{\gamma_{(n-1)}} \otimes \epsilon_{\beta_n} \otimes x^{\alpha_n} \epsilon_{\gamma_n}) \\
=& (\prod_{k<n} (-q_{k,n})^{\alpha_n \gamma_k-\beta_n(\alpha_k + \gamma_k)}(-1)^{\beta_n |\gamma_{(n-1)}|}) \\
& \hspace{5 mm} (\sum_{\ell=1}^{n-1} (-1)^{|\beta_{(n-1)}|} \delta_{\beta_{\ell+1}+...+\beta_{n-1}, 0} \delta_{\gamma_1+ ...+ \gamma_{\ell-1}, 0} \delta_{\alpha_{\ell}, 1} \prod_{\ell < k <n-1} (-q_{\ell, k})^{\alpha_k(\gamma_{\ell}+1)} \\
& \hspace{10 mm} \prod_{1 \leq k < \ell} (-q_{k,\ell})^{\alpha_k (\beta_{\ell} +1)} \prod_{\substack{1 \leq r < s \leq n-1 \\ r \neq \ell \neq s}} (-q_{r, s})^{\alpha_r(\alpha_s+\gamma_s)+\alpha_s \beta_r}  \\
& \hspace{10 mm} x_{\ell + 1}^{\alpha_{\ell+1}}...x_{n-1}^{\alpha_{n-1}} \epsilon_{\beta_{(n-1)} +\gamma_{(n-1)}+[\ell]} x_1^{\alpha_1}...x_{\ell-1}^{\alpha_{\ell-1}} \otimes \delta_{\beta_n, 0} x_n^{\alpha_n} \epsilon_{\gamma_n} \\
& \hspace{5 mm}+ (-1)^{|\beta_{(n-1)}| + |\gamma_{(n-1)}|} \delta_{|\gamma_{(n-1)}|, 0} \epsilon_{\beta_{(n-1)}} x^{\alpha_{(n-1)}} \otimes \delta_{\alpha_n, 1} (-1)^{\beta_n} \epsilon_{\beta_n+ \gamma_n +1} )
\end{align*}
by the inductive hypothesis.  Now, rearrange the terms to get the statement in the lemma.
\end{proof}




We now have all of the necessary pieces to compute the Gerstenhaber bracket, \begin{align} \label{phibrack} [f, g] = f \circ_{\tilde{\phi}} g - (-1)^{(m-1)(l-1)}g \circ_{\tilde{\phi}} f \end{align} for $f \in \Hom_{(\Lambda_{\q} \rtimes G)^e}((\tilde{\mathbb{K}})_m, \Lambda_{\q} \rtimes G)$ and $g \in \Hom_{(\Lambda_{\q} \rtimes G)^e}((\tilde{\mathbb{K}})_{l}, \Lambda_{\q} \rtimes G)$.  Notice $\{(x^{\alpha} \otimes g) (\epsilon_{\beta} \otimes 1)^*\}_{\alpha \in \{0,1\}^n, \beta \in \mathbb{N}^n, g \in G}$ forms a basis of $\textrm{Hom}_{(\Lambda_{\q} \rtimes G)^e}(\tilde{\mathbb{K}}, \Lambda_{\q} \rtimes G)$.  In the following theorem, we give the circle product on elements of this form.  While these elements are not necessarily non-zero elements of cohomology, the given formula can be extended linearly to give a well-defined bracket on cohomology by restricting to the elements of the form as in Theorem~\ref{vsthm}.

\begin{thm} \label{circthm} For $\alpha, \gamma \in \{0,1\}^n$, $\beta, \kappa \in \mathbb{N}^n$, and $g, h \in G$
\begin{align*}
(x^{\gamma} \otimes h)(\epsilon_{\kappa} \otimes 1)^* &\circ_{\tilde{\phi}} (x^{\alpha} \otimes g)(\epsilon_{\beta} \otimes 1)^* \\
=& \sum_{r=1}^n \sum_{\substack{\rho'+ \rho'' = \kappa+ \beta-[r]  \\ (\rho'-\beta)_{\ell} \geq 0 ~\forall \ell \in \{1,2,...,n\}}} (-1)^{|\rho'-\beta|(|\beta|+1)} \delta_{\rho'_{r+1}, \beta_{r+1}}...\delta_{\rho'_n, \beta_n} \delta_{\rho''_1+ ..+ \rho''_{r-1}, 0} \delta_{\alpha_r, 1}  \\
& \hspace{40 mm}  \mathbf{Q} (x^{\alpha + \gamma -[r]} \otimes hg)(\epsilon_{\kappa+\beta-[r]} \otimes 1)^*
\end{align*} 
where 
\begin{align*}
\mathbf{Q} = &\prod_{1 \leq s <r} \chi_{h,s}^{\alpha_s}(-q_{s, r})^{\alpha_s(\rho'_r-\beta_{r}+1)}  \prod_{1 \leq k < \ell < r \leq n} q_{k, \ell}^{\beta_k (\rho'-\beta)_{\ell}} \prod_{1 \leq r < k < \ell \leq n} q_{k, \ell}^{\rho''_k \beta_{\ell}}\prod_{r < s \leq n} (-q_{r, s})^{\alpha_s(\rho''_{r}+1)} \\
& \hspace{5 mm} \prod_{\substack{1 \leq t < u \leq n \\ t \neq r \neq u}} (-q_{t, u})^{\alpha_t(\alpha_u+\rho''_u)+\alpha_t (\rho'_u-\beta_u)} \prod_{\substack{1 \leq s <r \\s < v \leq n}} (-q_{s,v})^{-\alpha_s \gamma_v} \prod_{1 \leq v < r < s \leq n }(-q_{v,s})^{-(\gamma_v+\alpha_v)\alpha_s} \\
& \hspace{5 mm} \prod_{\substack{r < s \leq n \\ r \leq v < s}}(-q_{v,s})^{-\gamma_v \alpha_s}.
\end{align*}
\end{thm} 

\begin{proof}
\begin{align*}
 (x^{\gamma} \otimes h)(\epsilon_{\kappa} \otimes 1)^* &\circ_{\tilde{\phi}} (x^{\alpha} \otimes g)(\epsilon_{\beta} \otimes 1)^*(\epsilon_{\rho} \otimes 1) \\
=& (x^{\gamma} \otimes h)(\epsilon_{\kappa} \otimes 1)^* \tilde{\phi} (\mathds{1}_{\tilde{\mathbb{K}}} \otimes (x^{\alpha} \otimes g)(\epsilon_{\beta} \otimes 1)^* \otimes \mathds{1}_{\tilde{\mathbb{K}}} ) \tilde{\Delta}^{(2)} (\epsilon_{\rho} \otimes 1) \\
=& (x^{\gamma} \otimes h)(\epsilon_{\kappa} \otimes 1)^* \tilde{\phi} (\mathds{1}_{\tilde{\mathbb{K}}} \otimes (x^{\alpha} \otimes g)(\epsilon_{\beta} \otimes 1)^* \otimes \mathds{1}_{\tilde{\mathbb{K}}} ) (\tilde{\Delta} \otimes \mathds{1}_{\tilde{\mathbb{K}}}) \\
& \hspace{5 mm} \Big(\sum_{\rho'+ \rho'' = \rho} \prod_{\substack{1 \leq l \leq n \\ k<l}} q_{k, l}^{\rho''_k \rho'_{l}} \epsilon_{\rho'} \otimes \epsilon_{\rho''} \otimes 1\Big) \\
=& (x^{\gamma} \otimes h)(\epsilon_{\kappa} \otimes 1)^* \tilde{\phi} (\mathds{1}_{\tilde{\mathbb{K}}} \otimes (x^{\alpha} \otimes g)(\epsilon_{\beta} \otimes 1)^* \otimes \mathds{1}_{\tilde{\mathbb{K}}} ) \\
& \hspace{5 mm} \Big(\sum_{\nu'+ \nu'' = \rho'} \sum_{\rho'+ \rho'' = \rho} \prod_{\substack{1 \leq l \leq n \\ k<l}} q_{k, l}^{\rho''_k \rho'_{l}+ \nu''_k \nu'_{l}} \epsilon_{\nu'} \otimes \epsilon_{\nu''} \otimes \epsilon_{\rho''} \otimes 1\Big).
\end{align*}
In order to get a non-zero output from the function $\mathds{1}_{\tilde{\mathbb{K}}} \otimes (x^{\alpha} \otimes g)(\epsilon_{\beta} \otimes 1)^* \otimes \mathds{1}_{\tilde{\mathbb{K}}}$, we need $\nu''= \beta$.  Set $\nu''= \beta$, then $\nu'=\rho'-\beta$.  Applying this map thus gives us the Koszul sign $(-1)^{|\rho'-\beta||\beta|}$, making 
\begin{align*}
(x^{\gamma} \otimes h)(\epsilon_{\kappa} \otimes 1)^* &\circ_{\tilde{\phi}} (x^{\alpha} \otimes g)(\epsilon_{\beta} \otimes 1)^*(\epsilon_{\rho} \otimes 1) \\
=& (x^{\gamma} \otimes h)(\epsilon_{\kappa} \otimes 1)^* \tilde{\phi} \\
& \hspace{5 mm} \Big(\sum_{\substack{\rho'+ \rho'' = \rho \\ (\rho'-\beta)_{l} \geq 0 ~\forall l \in \{1,2,...,n\}}} (-1)^{|\rho'-\beta||\beta|} \prod_{\substack{1 \leq l \leq n \\ k<l}} q_{k, l}^{\rho''_k \rho'_{l}+ \beta_k (\rho'-\beta)_{l}} \\
&\hspace{40 mm} \epsilon_{\rho'-\beta} \otimes (x^{\alpha} \otimes g) \otimes \epsilon_{\rho''} \otimes 1\Big).
\end{align*}
We need $(\rho'-\beta)_{l} \geq 0$ for all $l \in \{1,2,...,n\}$ because $\epsilon_{\rho'-\beta}$ is tracking homological degree which is positive in each coordinate.  Therefore $\delta_{(\rho'-\beta)_{r+1}+ ... + (\rho'-\beta)_{n}, 0}=\delta_{\rho'_{r+1}, \beta_{r+1}}...\delta_{\rho'_n, \beta_n}$.  We will use this in the next expression.
\begin{align*}
(x^{\gamma} \otimes h)(\epsilon_{\kappa} \otimes 1)^* &\circ_{\tilde{\phi}} (x^{\alpha} \otimes g)(\epsilon_{\beta} \otimes 1)^*(\epsilon_{\rho} \otimes 1) \\
=& (x^{\gamma} \otimes h)(\epsilon_{\kappa} \otimes 1)^* \tilde{\phi} \\
& \hspace{5 mm} \Big(\sum_{\substack{\rho'+ \rho'' = \rho \\ (\rho'-\beta)_{l} \geq 0 ~\forall l \in \{1,2,...,n\}}} (-1)^{|\rho'-\beta||\beta|} \prod_{\substack{1 \leq l \leq n \\ k<l}} q_{k, l}^{\rho''_k \rho'_{l}+ \beta_k (\rho'-\beta)_{l}} \\
&\hspace{40 mm} \epsilon_{\rho'-\beta} \otimes (x^{\alpha} \otimes g) \otimes \epsilon_{\rho''} \otimes 1\Big) \\
=& (x^{\gamma} \otimes h)(\epsilon_{\kappa} \otimes 1)^* \\
& \hspace{5 mm} \Big(\sum_{\substack{\rho'+ \rho'' = \rho \\ (\rho'-\beta)_{l} \geq 0 ~\forall l \in \{1,2,...,n\}}} (-1)^{|\rho'-\beta||\beta|} \prod_{\substack{1 \leq l \leq n \\ k<l}} q_{k, l}^{\rho''_k \rho'_{l}+ \beta_k (\rho'-\beta)_{l}} \\
& \hspace{5 mm} \sum_{r=1}^n (-1)^{|\rho'-\beta|} \delta_{\rho'_{r+1}, \beta_{r+1}}...\delta_{\rho'_n, \beta_n} \delta_{\rho''_1+ ..+ \rho''_{r-1}, 0} \delta_{\alpha_r, 1} \prod_{r < s \leq n} (-q_{r, s})^{\alpha_s(\rho''_{r}+1)} \\
& \hspace{5 mm} \prod_{1 \leq s < r} (-q_{s, r})^{\alpha_s(\rho'_r-\beta_{r}+1)} \prod_{\substack{1 \leq t < u \leq n \\ t \neq r \neq u}} (-q_{t, u})^{\alpha_t(\alpha_u+\rho''_u)+\alpha_t (\rho'_u-\beta_u)} \\
& \hspace{5 mm} x_{r + 1}^{\alpha_{r+1}}...x_n^{\alpha_n} \epsilon_{\rho'-\beta+\rho''+[r]} x_1^{\alpha_1}...x_{r-1}^{\alpha_{r-1}} \otimes g\Big). 
\end{align*}
In order to get a non-zero output from the function $(x^{\gamma} \otimes h)(\epsilon_{\kappa} \otimes 1)^*$, we need $\rho'-\beta+\rho''+[r] = \kappa$.  That is, $\kappa + \beta - [r] = \rho' + \rho''= \rho$.  Notice $$x_{r + 1}^{\alpha_{r+1}}...x_n^{\alpha_n} \epsilon_{\rho'-\beta+\rho''+[r]} x_1^{\alpha_1}...x_{r-1}^{\alpha_{r-1}} \otimes g= (x_{r + 1}^{\alpha_{r+1}}...x_n^{\alpha_n} \otimes 1) (\epsilon_{\rho'-\beta+\rho''+[r]} \otimes 1)( x_1^{\alpha_1}...x_{r-1}^{\alpha_{r-1}} \otimes g)$$ by the definition of the multiplication on $\Lambda_{\q} \rtimes G$.  The second expression makes it clearer how to apply $(x^{\gamma} \otimes h)(\epsilon_{\kappa} \otimes 1)^*$.  Then 
\begin{align*}
(x^{\gamma} \otimes h)(\epsilon_{\kappa} \otimes 1)^* &\circ_{\tilde{\phi}} (x^{\alpha} \otimes g)(\epsilon_{\beta} \otimes 1)^*(\epsilon_{\rho} \otimes 1) \\
=& \sum_{\substack{\rho'+ \rho'' = \kappa+ \beta-[r]  \\ (\rho'-\beta)_{l} \geq 0 ~\forall l \in \{1,2,...,n\}}} (-1)^{|\rho'-\beta||\beta|} \prod_{\substack{1 \leq l \leq n \\ k<l}} q_{k, l}^{\rho''_k \rho'_{l}+ \beta_k (\rho'-\beta)_{l}} \\
& \hspace{5 mm} \sum_{r=1}^n (-1)^{|\rho'-\beta|} \delta_{\rho'_{r+1}, \beta_{r+1}}...\delta_{\rho'_n, \beta_n} \delta_{\rho''_1+ ..+ \rho''_{r-1}, 0} \delta_{\alpha_r, 1} \prod_{r < s \leq n} (-q_{r, s})^{\alpha_s(\rho''_{r}+1)} \\
& \hspace{5 mm} \prod_{1 \leq s < r} (-q_{s, r})^{\alpha_s(\rho'_r-\beta_{r}+1)} \prod_{\substack{1 \leq t < u \leq n \\ t \neq r \neq u}} (-q_{t, u})^{\alpha_t(\alpha_u+\rho''_u)+\alpha_t (\rho'_u-\beta_u)} \\
& \hspace{5 mm} (x_{r + 1}^{\alpha_{r+1}}...x_n^{\alpha_n} \otimes 1)(x^{\gamma} \otimes h)( x_1^{\alpha_1}...x_{r-1}^{\alpha_{r-1}} \otimes g) \\
=&  \sum_{r=1}^n \sum_{\substack{\rho'+ \rho'' = \kappa+ \beta-[r]  \\ (\rho'-\beta)_{l} \geq 0 ~\forall l \in \{1,2,...,n\}}} (-1)^{|\rho'-\beta|(|\beta|+1)} \prod_{\substack{1 \leq l \leq n \\ k<l}} q_{k, l}^{\rho''_k \rho'_{l}+ \beta_k (\rho'-\beta)_{l}} \\
& \hspace{5 mm}  \delta_{\rho'_{r+1}, \beta_{r+1}}...\delta_{\rho'_n, \beta_n} \delta_{\rho''_1+ ..+ \rho''_{r-1}, 0} \delta_{\alpha_r, 1} \prod_{r < s \leq n} (-q_{r, s})^{\alpha_s(\rho''_{r}+1)} \\
& \hspace{5 mm}\prod_{1 \leq s < r} (-q_{s, r})^{\alpha_s(\rho'_r-\beta_{r}+1)} \prod_{\substack{1 \leq t < u \leq n \\ t \neq r \neq u}} (-q_{t, u})^{\alpha_t(\alpha_u+\rho''_u)+\alpha_t (\rho'_u-\beta_u)} \prod_{1 \leq s <r} \chi_{h,s}^{\alpha_s} \\
& \hspace{5 mm} \prod_{\substack{1 \leq s <r \\s < v \leq n}} (-q_{s,v})^{-\alpha_s \gamma_v} \prod_{\substack{r < s \leq n \\ 1 \leq v <r }}(-q_{v,s})^{-(\gamma_v+\alpha_v)\alpha_s} \prod_{\substack{r < s \leq n \\ r \leq v < s}}(-q_{v,s})^{-\gamma_v \alpha_s} \\
& \hspace{30 mm} x^{\alpha + \gamma -[r]} \otimes hg.
\end{align*}
This expression can be simplified slightly to eliminate trivial terms.  That is, 
\begin{align*}
(x^{\alpha} \otimes g)(\epsilon_{\beta} \otimes 1)^* &\circ_{\tilde{\phi}} (x^{\gamma} \otimes h)(\epsilon_{\kappa} \otimes 1)^* \\
=& \sum_{r=1}^n \sum_{\substack{\rho'+ \rho'' = \kappa+ \beta-[r]  \\ (\rho'-\beta)_{\ell} \geq 0 ~\forall \ell \in \{1,2,...,n\}}} (-1)^{|\rho'-\beta|(|\beta|+1)} \delta_{\rho'_{r+1}, \beta_{r+1}}...\delta_{\rho'_n, \beta_n} \delta_{\rho''_1+ ..+ \rho''_{r-1}, 0} \delta_{\alpha_r, 1}  \\
& \hspace{5 mm} \prod_{1 \leq s <r} \chi_{h,s}^{\alpha_s}(-q_{s, r})^{\alpha_s(\rho'_r-\beta_{r}+1)}  \prod_{1 \leq k < \ell < r \leq n} q_{k, \ell}^{\beta_k (\rho'-\beta)_{\ell}} \prod_{1 \leq r < k < \ell \leq n} q_{k, \ell}^{\rho''_k \beta_{\ell}} \\
& \hspace{5 mm} \prod_{r < s \leq n} (-q_{r, s})^{\alpha_s(\rho''_{r}+1)} \prod_{\substack{1 \leq t < u \leq n \\ t \neq r \neq u}} (-q_{t, u})^{\alpha_t(\alpha_u+\rho''_u)+\alpha_t (\rho'_u-\beta_u)} \prod_{\substack{1 \leq s <r \\s < v \leq n}} (-q_{s,v})^{-\alpha_s \gamma_v} \\
& \hspace{5 mm} \prod_{1 \leq v < r < s \leq n }(-q_{v,s})^{-(\gamma_v+\alpha_v)\alpha_s} \prod_{\substack{r < s \leq n \\ r \leq v < s}}(-q_{v,s})^{-\gamma_v \alpha_s} \\
& \hspace{30 mm} (x^{\alpha + \gamma -[r]} \otimes hg)(\epsilon_{\kappa+\beta-[r]} \otimes 1)^*.
\end{align*}

\end{proof}


\subsection{Example in 2 generators} \label{ex}

Now that we have the general formulas for the Gerstenhaber algebra structure on $\HH^*(\Lambda_{\mathbf{q}}^n \rtimes G)$, we can apply them to a simple example very similar to the example computed in \cite[5.1]{GNW}.  Let $n=2$ and assume $q_{1,2}$ is not a root of unity.  For simplicity, let $q=q_{1,2}$.

Then, by Theorem \ref{vsthm}, $$\HH^m(\Lambda_{\textbf{q}}^2 \rtimes G) \cong (\bigoplus_{\substack{\beta \in \mathbb{N}^n \\ |\beta|=m}} \bigoplus_{\substack{\alpha \in \{0,1\}^n \\ \beta-\alpha \in C_g}} span_k \{(x^{\alpha} \otimes g) \epsilon_{\beta}^*\})^G.$$ where $C_g = \{\gamma \in (\mathbb{N} \cup \{-1\})^2 | \forall i, \gamma_i=-1 \textrm{ or } (-1)^{\gamma_i}\prod_{k \neq i} (-q_{k,l})^{\gamma_k} = \chi_{g,i} \}.$

Therefore we have two conditions on the $\gamma=\beta-\alpha$ for which $(x^{\alpha} \otimes g) \epsilon_{\beta}^*$ is non-trivial in $\HH^m(\Lambda_{\textbf{q}}^2, \Lambda_{\textbf{q}}^2 \rtimes G)$,
\begin{align*}
\gamma_1= -1 \textrm{ or }& (-1)^{\gamma_1}(-q^{-1})^{\gamma_2} = \chi_{g,1} \\
\textrm{ and } \gamma_2= -1 \textrm{ or }& (-1)^{\gamma_2}(-q)^{\gamma_1} = \chi_{g,2}.
\end{align*}

If $\gamma_1=-1$, then, as $q$ is not a root of unity and $\chi_{g,2}$ must be a root of unity, we cannot have $(-q)^{-1} = (-1)^{\gamma_2}\chi_{g,2}$ and thus, for $\gamma \in C_g$, we need $\gamma_2 = -1$.

Alternatively, if $\gamma_1 \neq -1$, then for $\gamma \in C_g$, we need $(-1)^{\gamma_1}(-q^{-1})^{\gamma_2} = \chi_{g,1}$.  But, again because $q$ is not a root of unity, we must have $\gamma_2=0$ and $\chi_{g,1}=1$.  Thus $\gamma_1 \neq -1$ forces $\gamma_2 \neq -1$ also.  Therefore, for $\gamma \in C_g$, $\gamma_1$ must satisfy $(-1)^{\gamma_2}(-q)^{\gamma_1} = \chi_{g,2}$, forcing $\gamma_1=0$ and $\chi_{g,2}=1$.

That is, we have two options for non-trivial elements, either $\gamma=(-1,-1)$ or $\gamma=(0,0)$ and $\chi_{g,1}=\chi_{g,2}=1$, making 
\begin{align*}
\HH^*(\Lambda_{\textbf{q}}^2 \rtimes G) &\cong ( span_k \{\epsilon_{0,0}^*, \{(x_1x_2 \otimes g) \epsilon_{0,0}^*\}_{g \in G} , \\
& \hspace{20 mm}  \{(x_2 \otimes g)\epsilon_{0,1}^*, (x_1 \otimes g)\epsilon_{1,0}^*, (x_1 x_2 \otimes g)\epsilon_{1,1}^* \}_{\substack{g \in G \\ \chi_{g,1}=\chi_{g,2}=1}} \})^G \\
&\cong  span_k \{ \epsilon_{0,0}^*, \{(x_1x_2 \otimes g) \epsilon_{0,0}^*\}_{\substack{g \in Z(G) \\ \chi_{g,1}\chi_{g,2}=1}} ,\\
& \hspace{20 mm} \{(x_2 \otimes g)\epsilon_{0,1}^*, (x_1 \otimes g)\epsilon_{1,0}^*, (x_1 x_2 \otimes g)\epsilon_{1,1}^* \}_{\substack{g \in Z(G) \\ \chi_{g,1}=\chi_{g,2}=1}} \}.
\end{align*}

We can now use our formula from Theorem \ref{cupthm}, $$(x^{\alpha} \otimes g) \epsilon_{\beta}^* \smile (x^{\gamma} \otimes h) \epsilon_{\kappa}^* = \prod_{l=1}^2 \chi_{g,l}^{\gamma_l} \prod_{k<l} q_{k,l}^{\kappa_k \beta_l-\gamma_k \alpha_l}(-1)^{-\gamma_k \alpha_l} (x^{\alpha+\gamma} \otimes gh) \epsilon_{\beta +\kappa}^*$$ to compute cup products.  Because of the $x^{\alpha + \gamma}$ in the result of the product, the only possibly non-zero cup products are
\begin{align*}
(x_2 \otimes g) \epsilon_{0,1}^* \smile (x_1 \otimes h) \epsilon_{1,0}^* &= \chi_{g,1}^1 q^{1-1}(-1)^1(x_1 x_2 \otimes gh) \epsilon_{1,1}^* \\
&=-\chi_{g,1}(x_1 x_2 \otimes gh) \epsilon_{1,1}^* \\
&=-(x_1 x_2 \otimes gh) \epsilon_{1,1}^*
\end{align*}
and, using either the formula or the anti-commutativity of $\smile$ on $\HH^m(\Lambda_{\mathbf{q}}^2 \rtimes G)$, $$(x_1 \otimes h) \epsilon_{0,1}^* \smile (x_2 \otimes g) \epsilon_{1,0}^* = (x_1 x_2 \otimes gh) \epsilon_{1,1}^*.$$

Finally, we can use our formula from Theorem~\ref{circthm}, restated for the case $n=2$, 
\begin{align*}
(x^{\gamma} \otimes h)(\epsilon_{\kappa} \otimes 1)^*  &\circ (x^{\alpha} \otimes g)(\epsilon_{\beta} \otimes 1)^*\\
=& \sum_{r=1}^2 \sum_{\substack{\rho'+ \rho'' = \kappa+ \beta-[r]  \\ (\rho'-\beta)_{\ell} \geq 0 ~\forall \ell \in \{1,2\}}} (-1)^{|\rho'-\beta|(|\beta|+1)} \delta_{\rho'_{r+1}, \beta_{r+1}}...\delta_{\rho'_n, \beta_n} \delta_{\rho''_1+ ..+ \rho''_{r-1}, 0} \delta_{\alpha_r, 1}  \\
& \hspace{5 mm} \prod_{1 \leq s <r} \chi_{h,s}^{\alpha_s}(-q_{s, r})^{\alpha_s(\rho'_r-\beta_{r}+1)} \prod_{r < s \leq 2} (-q_{r, s})^{\alpha_s(\rho''_{r}+1)} (x^{\alpha + \gamma -[r]} \otimes hg) \\
& \hspace{30 mm} (\epsilon_{\kappa+\beta-[r]} \otimes 1)^*
\end{align*}
to compute brackets.  Then the non-zero $\circ$-products are 
\begin{align*}
(x_2 \otimes h)(\epsilon_{0,1} \otimes 1)^*  &\circ (x_1 x_2 \otimes g)(\epsilon_{0,0} \otimes 1)^* \\
& \hspace{5 mm} =\sum_{\rho'+\rho''=(0,0)}\chi_{h,1}(-q)^{1(0-0+1)}(-q)^{-1} (x_1 x_2 \otimes hg) (\epsilon_{0,0} \otimes 1)^* \\
&=\chi_{h,1}(x_1 x_2 \otimes hg)(\epsilon_{0,0} \otimes 1)^* \\
&=(x_1 x_2 \otimes hg)(\epsilon_{0,0} \otimes 1)^*
\end{align*} 
and 
\begin{align*}
(x_2 \otimes h)(\epsilon_{0,1} \otimes 1)^* &\circ (x_1 x_2 \otimes g)(\epsilon_{0,0} \otimes 1)^*= (x_1 x_2 \otimes hg)(\epsilon_{0,0} \otimes 1)^*, \\
(x_1 \otimes h)(\epsilon_{1,0} \otimes 1)^* &\circ (x_1 \otimes g)(\epsilon_{1,0} \otimes 1)^* =(x_1 \otimes hg)(\epsilon_{1,0} \otimes 1)^*, \\
(x_1 \otimes h)(\epsilon_{1,0} \otimes 1)^* &\circ (x_1 x_2 \otimes g)(\epsilon_{0,0} \otimes 1)^* =(x_1 x_2 \otimes hg)(\epsilon_{0,0} \otimes 1)^*, \\
(x_1 x_2 \otimes h)(\epsilon_{1,1} \otimes 1)^* &\circ (x_1 \otimes g)(\epsilon_{1,0} \otimes 1)^* =(x_1 x_2 \otimes hg)(\epsilon_{1,1} \otimes 1)^*, \\
(x_1 \otimes g)(\epsilon_{1,0} \otimes 1)^* &\circ (x_1 x_2 \otimes h)(\epsilon_{1,1} \otimes 1)^* =(x_1 x_2 \otimes gh)(\epsilon_{1,1} \otimes 1)^*, \\
(x_2 \otimes h)(\epsilon_{0,1} \otimes 1)^* &\circ (x_2 \otimes g)(\epsilon_{0,1} \otimes 1)^* =(x_2 \otimes hg)(\epsilon_{0,1} \otimes 1)^*, \\
(x_1 x_2 \otimes h)(\epsilon_{1,1} \otimes 1)^* &\circ (x_2 \otimes g)(\epsilon_{0,1} \otimes 1)^* =(x_1 x_2 \otimes hg)(\epsilon_{1,1} \otimes 1)^*, \textrm{ and}\\
(x_2 \otimes g)(\epsilon_{0,1} \otimes 1)^* &\circ (x_1 x_2 \otimes h)(\epsilon_{1,1} \otimes 1)^* =(x_1 x_2 \otimes gh)(\epsilon_{1,1} \otimes 1)^*. 
\end{align*}

Using our formula (\ref{phibrack}), modifying for this notation, \begin{align*} [(x^{\alpha} \otimes g)(\epsilon_{\beta} \otimes 1)^*,& (x^{\gamma} \otimes h)(\epsilon_{\kappa} \otimes 1)^*]\\
&= (x^{\alpha} \otimes g)(\epsilon_{\beta} \otimes 1)^* \circ (x^{\gamma} \otimes h)(\epsilon_{\kappa} \otimes 1)^* \\
& \hspace{5 mm} - (-1)^{(|\kappa|-1)(|\beta|-1)} (x^{\gamma} \otimes h)(\epsilon_{\kappa} \otimes 1)^* \circ (x^{\alpha} \otimes g)(\epsilon_{\beta} \otimes 1)^*,
\end{align*} we can complete the bracket computations.  The non-zero brackets among pairs of the generators of $\HHD(\Lambda_{\mathbf{q}}^2 \rtimes G)$ are 
\begin{align*}
[(x_2 \otimes h)(\epsilon_{0,1} \otimes 1)^* ,& (x_1 x_2 \otimes g)(\epsilon_{0,0} \otimes 1)^*] = (x_1 x_2 \otimes hg)(\epsilon_{0,0} \otimes 1)^* \textrm{ and} \\
[(x_1 \otimes h)(\epsilon_{1,0} \otimes 1)^* ,& (x_1 x_2 \otimes g)(\epsilon_{0,0} \otimes 1)^*] = (x_1 x_2 \otimes hg) (\epsilon_{0,0} \otimes 1)^*. 
\end{align*}

Notice that the bracket and cup product computations agree with the results in \cite[Section~5.1]{GNW} and \cite[Section~2.1]{BGMS} respectively when $G=1$.


\section{A specific group action on $\Lambda_{\q}^2$}\label{equal}

We will now consider a specific example, on $\Lambda_{\q}^2$, to study the structure in more depth.  As in the previous section, to ease notation, let $q_{1,2}=q$.  Let $G=\langle g \rangle$ be the cyclic group generated by an element $g$.  Assume $|G|=m$ and that $G$ acts on $\Lambda_{\q}^2$ by $^g x_1 = q x_1$ and $^g x_2 = q^{-1} x_2$.  That is, $G$ acts on $\Lambda_{\q}^2$ by its quantum coefficient.  And let $q$ be a primitive $d$th root of unity.  

By Theorem \ref{vsthm}, we know that $\HH^m(\Lambda_{\q}^2 \rtimes G)$ is isomorphic to the $G$-invariants of $$\bigoplus_{i=1}^d \bigoplus_{\substack{\beta \in \mathbb{N}^2 \\ |\beta|=d}} \bigoplus_{\substack{\alpha \in \{0,1\}^2 \\ \beta-\alpha \in C_{g^i}}} span_k \{(x^{\alpha} \otimes g^i) \epsilon_{\beta}^*\}$$ where 
\begin{align*} C_{g^i}=& \{\gamma \in (\N \cup \{-1\})^2 | ~\forall~l, \gamma_l =-1 \textrm{ or, for } k \neq l, (-1)^{\gamma_l} (-q_{k, l})^{\gamma_k} = q_{l,k}^i \} \\
=& \{\gamma \in (\N \cup \{-1\})^2 | ~\forall~l, \gamma_l =-1 \textrm{ or, for } k \neq l, q_{k, l}^{\gamma_k+i} = (-1)^{|\gamma|} \} \\
=& \{\gamma \in (\N \cup \{-1\})^2 | ~\gamma_1 =-1 \textrm{ or } q^{\gamma_2+i} = (-1)^{|\gamma|} \textrm{ and }\\
& \hspace{35 mm} \gamma_2 =-1 \textrm{ or } q^{\gamma_1+i} = (-1)^{|\gamma|} \}. \\
\end{align*}

Therefore to have a non-trivial element of cohomology, we need to first check the conditions for which $\gamma \in (\N \cup \{-1\})^2$ satisfies $\gamma_1 =-1$ or $q^{\gamma_2+i} = (-1)^{|\gamma|}$ and $\gamma_2 =-1$ or $q^{\gamma_1+i} = (-1)^{|\gamma|}$ and then check for terms that are $G$-invariant.

To compute the cup product, by Theorem \ref{cupthm}, we know that for $\alpha, \gamma \in \{0,1\}^2$, $\beta, \kappa \in \mathbb{N}^2$, and $g^i, g^j \in G$,  
\begin{align*} (x^{\alpha} \otimes g^i) \epsilon_{\beta}^* \smile (x^{\gamma} \otimes g^j) \epsilon_{\kappa}^* =& \prod_{l=1}^2 \prod_{l \neq k} q_{l,k}^{i\gamma_l} \prod_{k<l} q_{k,l}^{\kappa_k \beta_l-\gamma_k \alpha_l}(-1)^{-\gamma_k \alpha_l} (x^{\alpha+\gamma} \otimes g^{i+j}) \epsilon_{\beta +\kappa}^* \\
=& q^{i(\gamma_1-\gamma_2) + \kappa_1 \beta_2-\gamma_1 \alpha_2}(-1)^{-\gamma_1 \alpha_2} (x^{\alpha+\gamma} \otimes g^{i+j}) \epsilon_{\beta +\kappa}^*. \\
\end{align*}

And, under these assumptions, Theorem \ref{circthm} becomes, for $\alpha, \gamma \in \{0,1\}^2$, $\beta, \kappa \in \mathbb{N}^2$, and $g^i, g^j \in G$
\begin{align*}
(x^{\alpha} \otimes g^i)(\epsilon_{\beta} \otimes 1)^* &\circ_{\tilde{\phi}} (x^{\gamma} \otimes g^j)(\epsilon_{\kappa} \otimes 1)^* \\
=& \sum_{r=1}^2 \sum_{\substack{\rho'+ \rho'' = \kappa+ \beta-[r]  \\ (\rho'-\beta)_{\ell} \geq 0 ~\forall \ell \in \{1,2\}}} (-1)^{|\rho'-\beta|(|\beta|+1)} \delta_{\rho'_{r+1}, \beta_{r+1}}...\delta_{\rho'_n, \beta_n} \delta_{\rho''_1+ ..+ \rho''_{r-1}, 0} \delta_{\alpha_r, 1}  \\
& \hspace{5 mm} \prod_{1 \leq s <r} q_{s,r}^{j\alpha_s}(-q_{s, r})^{\alpha_s(\rho'_r-\beta_{r}+1)} \prod_{r < s \leq 2} (-q_{r, s})^{\alpha_s(\rho''_{r}+1)} \\
& \hspace{5 mm} \prod_{\substack{1 \leq s <r \\s < v \leq 2}} (-q_{s,v})^{-\alpha_s \gamma_v} (x^{\alpha + \gamma -[r]} \otimes g^{i+j})(\epsilon_{\kappa+\beta-[r]} \otimes 1)^* \\
=& \sum_{\substack{\rho'+ \rho'' = \kappa+ \beta-[1]  \\ (\rho'-\beta)_{\ell} \geq 0 ~\forall \ell \in \{1,2\}}} (-1)^{|\rho'-\beta|(|\beta|+1)} \delta_{\rho'_{2}, \beta_{2}} \delta_{\alpha_1, 1} (-q_{1, 2})^{\alpha_2(\rho''_{1}+1)}  \\
& \hspace{10 mm} (x^{\alpha + \gamma -[1]} \otimes g^{i+j})(\epsilon_{\kappa+\beta-[1]} \otimes 1)^* \\
& \hspace{5 mm} + \sum_{\substack{\rho'+ \rho'' = \kappa+ \beta-[2]  \\ (\rho'-\beta)_{\ell} \geq 0 ~\forall \ell \in \{1,2\}}} (-1)^{|\rho'-\beta|(|\beta|+1)} \delta_{\rho''_1, 0} \delta_{\alpha_2, 1} q^{j\alpha_2}(-q)^{\alpha_1(\rho'_2-\beta_{2}+1)-\alpha_1 \gamma_2} \\
& \hspace{10 mm}  (x^{\alpha + \gamma -[2]} \otimes g^{i+j})(\epsilon_{\kappa+\beta-[2]} \otimes 1)^* .\\
\end{align*} 

As in \cite{BGMS} and \cite{GNW}, we study the specific structure of Hochschild cohomology for each choice of $d$.

\subsection{$d>1$ is odd}

If $d>1$ is odd, then 
\begin{align*}
C_{g^i} =& \{\gamma \in (\N \cup \{-1\})^2 | ~\gamma_1 =-1 \textrm{ or } (\exists~ t, k \in \N, \gamma_2+i = td \textrm{ and } |\gamma| = 2k) \textrm{ and }\\
& \hspace{35 mm} \gamma_2 =-1 \textrm{ or } (\exists~t, k \in \N, \gamma_1+i= td \textrm{ and } |\gamma| = 2k) \} \\
=& \{\gamma \in (\N \cup \{-1\})^2 | ~\gamma_1 =-1 \textrm{ and } \gamma_2=-1 \textrm{ or }\\
& \hspace{35 mm} \gamma_1 =-1 \textrm{ and } \exists~ t \in \N, \gamma_2 = 2t+1 \textrm{ and } i=1 \textrm{ or } \\
& \hspace{35 mm} \gamma_2 =-1 \textrm{ and } \exists~ t \in \N, \gamma_1 = 2t+1 \textrm{ and } i=1 \textrm{ or } \\
& \hspace{35 mm} \exists~ t, t' \in \N, \gamma_1 = td-i \textrm{ and } \gamma_2 = t'd-i \textrm{ and } t+t' \textrm{ even}\}. \\
\end{align*}

From this description, we immediately get  
\begin{align*} HH^*(\Lambda_{\q}^2, \Lambda_{\q}^2 \rtimes G) \cong & \bigoplus_{t \in \N} span_k \{(x_1 \otimes g) \epsilon_{0,2t+1}^*, (x_2 \otimes g) \epsilon_{2t+1, 0}^*, (x_1 x_2 \otimes g)\epsilon_{2t, 0}^*, \\
& \hspace{10 mm} (x_1 x_2 \otimes g) \epsilon_{0,2t}^* \} \\
&\bigoplus_{i=1}^d \bigoplus_{\substack{ t, t' \in \N \\ t+ t' \textrm{ even}}} span_k\{ (1 \otimes g^i) \epsilon_{td-i, t'd-i}^*, (x_1 \otimes g^i) \epsilon_{td-i+1, t'd-i}^*, \\
& \hspace{20 mm} (x_2 \otimes g^i) \epsilon_{td-i, t'd-i+1}^*, (x_1 x_2 \otimes g^i) \epsilon_{td-i+1, t'd-i+1}^* \} \\
& \bigoplus_{i=1}^d span_k \{ (x_1 x_2 \otimes g^i)\epsilon_{0,0}^*, (1 \otimes g^i) \epsilon_{0,0}^* \} \\
\end{align*} as a vector space.

Then, after taking $G$-invariants, we get
\begin{align*} HH^*(\Lambda_{\q}^2 \rtimes G) \cong & \bigoplus_{t \in \N} span_k \{(x_1 \otimes g) \epsilon_{0,2td-1}^*, (x_2 \otimes g) \epsilon_{2td-1, 0}^*, (x_1 x_2 \otimes g)\epsilon_{2td, 0}^*, (x_1 x_2 \otimes g) \epsilon_{0,2td}^* \} \\
&\bigoplus_{i=1}^d \bigoplus_{\substack{ t, t' \in \N \\ t+ t' \textrm{ even}}} span_k\{ (1 \otimes g^i) \epsilon_{td-i, t'd-i}^*, (x_1 \otimes g^i) \epsilon_{td-i+1, t'd-i}^*, \\
& \hspace{20 mm} (x_2 \otimes g^i) \epsilon_{td-i, t'd-i+1}^*, (x_1 x_2 \otimes g^i) \epsilon_{td-i+1, t'd-i+1}^* \} \\
& \bigoplus_{i=1}^d span_k \{ (x_1 x_2 \otimes g^i)\epsilon_{0,0}^*, (1 \otimes g^i) \epsilon_{0,0}^* \} \\
\end{align*} as a vector space.

We will forgo the cup product structure in favor of showing the bracket structure.  The non-zero brackets are, for $t, t', t'', t''' \in \N$ and $i, j \in \{1, 2, ..., d\}$,
\begin{flalign*} 
&[(x_2 \otimes g) \epsilon_{2t'd-1, 0}^*, (x_1 \otimes g) \epsilon_{0,2td-1}^*] = (2t'd-1)(x_2 \otimes g^2) \epsilon_{2t'd-2, 2td-1}^*,&\\
&\hspace{65 mm} -(2td-1)q(x_1 \otimes g^2) \epsilon_{2t'd-1, 2td-2}^* ,& \\
&[(x_1 x_2 \otimes g) \epsilon_{2t'd, 0}^*, (x_1 \otimes g) \epsilon_{0,2td-1}^*] = (2t'd)(x_1 x_2 \otimes g^2) \epsilon_{2t'd-1, 2td-1}^*,&\\
&[(1 \otimes g^i) \epsilon_{t'd-i, t''d-i}^*, (x_1 \otimes g) \epsilon_{0,2td-1}^*] = (t'd-i)(1 \otimes g^{i+1}) \epsilon_{t'd-i-1, 2td+t''d-i-1}^* ,&\\
&[(x_1 \otimes g^i) \epsilon_{t'd-i+1, t''d-i}^*, (x_1 \otimes g) \epsilon_{0,2td-1}^*] = (t'd-i+1)(x_1 \otimes g^{i+1}) \epsilon_{t'd-i, 2td+t''d-i-1}^* ,&\\
&[(x_2 \otimes g^i) \epsilon_{t'd-i, t''d-i+1}^*, (x_1 \otimes g) \epsilon_{0,2td-1}^*] =  (t'd-i)(x_2 \otimes g^{i+1}) \epsilon_{t'd-i-1, 2td+t''d-i}^*&\\
&\hspace{70 mm}-(2td-1)q(x_1 \otimes g^{i+1}) \epsilon_{t'd-i, 2td+t''d-i}^* ,&\\
&[(x_1 x_2 \otimes g^i) \epsilon_{t'd-i+1, t''d-i+1}^*, (x_1 \otimes g) \epsilon_{0,2td-1}^*] = (t'd-i+1)(x_1 x_2 \otimes g^{i+1}) \epsilon_{t'd-i, 2td+t''d-i}^* ,&\\
&[(x_1 x_2 \otimes g^i) \epsilon_{0, 2t'd}^*, (x_2 \otimes g) \epsilon_{0,2td-1}^*] = (2t'd)q (x_1 x_2 \otimes g^2) \epsilon_{2td-1, 2t'd-1}^* ,&\\
&[(1 \otimes g^i) \epsilon_{t'd-i, t''d-i}^*, (x_2 \otimes g) \epsilon_{0,2td-1}^*] = (t''d-i) q^i (1 \otimes g^{i+1}) \epsilon_{2td+t'd-i-1, t''d-i-1}^* ,&\\
&[(x_1 \otimes g^i) \epsilon_{t'd-i+1, t''d-i}^*, (x_2 \otimes g) \epsilon_{0,2td-1}^*] = (t''d-i)q^i (x_1 \otimes g^{i+1}) \epsilon_{2td+t'd-i, t''d-i-1}^* &\\
& \hspace{75 mm} -(2td-1)(x_2 \otimes g^{i+1}) \epsilon_{2td+t'd-i-1, t''d-i}^*, & \\
&[(x_2 \otimes g^i) \epsilon_{t'd-i, t''d-i+1}^*, (x_2 \otimes g) \epsilon_{0,2td-1}^*] = (t''d-i+1) q^i (x_2 \otimes g^{i+1}) \epsilon_{2td+t'd-i-1, t''d-i}^*,&\\
&[(x_1 x_2 \otimes g^i) \epsilon_{t'd-i+1, t''d-i+1}^*, (x_2 \otimes g) \epsilon_{0,2td-1}^*] \\
&\hspace{30 mm}= (t''d-i+1)q^i (x_1 x_2 \otimes g^{i+1}) \epsilon_{2td+t'd-i, t''d-i}^* ,&\\
&[(1 \otimes g^i) \epsilon_{t'd-i, t''d-i}^*, (x_1 x_2 \otimes g) \epsilon_{2td,0}^*] &\\
&\hspace{30 mm} = (-1)^{t'd-i+1} \sum_{\rho'_2=0}^{ t''d-i-1} q^{\rho'_2 +i + 1}(x_1 \otimes g^{i+1}) \epsilon_{2td+t'd-i, t''d-i-1}^* ,&\\
&[(x_1 \otimes g^i) \epsilon_{t'd-i+1, t''d-i}^*, (x_1 x_2 \otimes g) \epsilon_{2td,0}^*]  &\\
&\hspace{30 mm} = -(2td) (x_1 x_2 \otimes g^{i+1}) \epsilon_{2td+t'd-i, t''d-i}^* ,&\\
&[(x_2 \otimes g^i) \epsilon_{t'd-i, t''d-i+1}^*, (x_1 x_2 \otimes g) \epsilon_{2td,0}^*] &\\ 
&\hspace{30 mm} = (-1)^{t'd-i} \sum_{\rho'_2 = 0}^{t''d-i} q^{\rho'_2 +i} (x_1 x_2 \otimes g^{i+1}) \epsilon_{2td+t'd-i, t''d-i}^* ,&\\
&[(1 \otimes g^i) \epsilon_{t'd-i, t''d-i}^*, (x_1 x_2 \otimes g) \epsilon_{0,2td}^*] &\\
&\hspace{30 mm} = (-1)^{t'd-i} \sum_{\rho''_1 = 0}^{t'd -i-1 } q^{\rho''_1+1} (x_2 \otimes g^{i+1}) \epsilon_{t'd-i-1, 2td+t''d-i}^* & \\
&\hspace{35 mm} -(-1)^{t'd-i} \sum_{\rho'_2=2td}^{t''d-i-1} q^{\rho'_2 -2td+i + 1}(x_1 \otimes g^{i+1}) \epsilon_{t'd-i, 2td+t''d-i-1}^* ,&\\
&[(x_1 \otimes g^i) \epsilon_{t'd-i+1, t''d-i}^*, (x_1 x_2 \otimes g) \epsilon_{0,2td}^*] &\\
&\hspace{30 mm} = (-1)^{t'd-i+1} \sum_{\rho''_1 = 0}^{t'd -i } q^{\rho''_1+1} (x_1 x_2 \otimes g^{i+1}) \epsilon_{t'd-i-1, 2td+t''d-i}^* ,& \\
&[(x_2 \otimes g^i) \epsilon_{t'd-i, t''d-i+1}^*, (x_1 x_2 \otimes g) \epsilon_{0,2td}^*] & \\ 
&\hspace{30 mm} =(-1)^{t'd-i} \sum_{\rho'_2= 2td}^{2td+ t''d-i} q^{\rho'_2 -2td+i }(x_1 x_2 \otimes g^{i+1}) \epsilon_{t'd-i, 2td+t''d-i}^* &\\
& \hspace{35 mm} -(2td)q (x_1 x_2 \otimes g^{i+1}) \epsilon_{t'd-i, 2td+t''d-i}^* ,&\\
&[(x_1 \otimes g^j) \epsilon_{t''d-j+1, t'''d-j}^*, (1 \otimes g^i) \epsilon_{td-i,t'd-i}^*] \\
&\hspace{30 mm}= -(td-i)(1 \otimes g^{i+j}) \epsilon_{(t+t'')d-i-j, (t'+t''')d-i-j}^* ,&\\
&[(x_2 \otimes g^j) \epsilon_{t'd-j, t''d-j+1}^*, (1 \otimes g^i) \epsilon_{td-i,t'd-i}^*] \\
&\hspace{30 mm}= -(t'd-i)q^i(1 \otimes g^{i+j}) \epsilon_{(t+t'')d-i-j, (t'+t''')d-i-j}^* ,&\\
&[(x_1 x_2 \otimes g^j) \epsilon_{t'd-j+1, t''d-j+1}^*, (1 \otimes g^i) \epsilon_{td-i,t'd-i}^*]  &\\
&\hspace{10 mm} = (-1)^{t''d-j+1} \sum_{\rho''_1 = 0}^{td -j-1} q^{\rho''_1+1} (x_2 \otimes g^{i+j}) \epsilon_{(t+t'')d-i-j, (t'+t''')d-i-j+1}^* & \\
&\hspace{15 mm} +(-1)^{td-i} \sum_{\rho'_2=t'''d-i+1}^{(t'+t''')d-i-j} q^{\rho'_2-t'''d +i + j}(x_1 \otimes g^{i+j}) \epsilon_{(t+t'')d-i-j+1, (t'+t''')d-i-j}^* ,&\\
&[(x_1 x_2 \otimes g^j) \epsilon_{0,0}^*, (1 \otimes g^i) \epsilon_{td-i,t'd-i}^*] &\\
&\hspace{30 mm} = (-1)^{td-i} \sum_{\rho''_1 =0}^{td -i-1 } q^{\rho''_1+1} (x_2 \otimes g^{i+j}) \epsilon_{td-i-1, t'd-i}^* & \\
&\hspace{35 mm} +(-1)^{td-i} \sum_{\rho'_2= 0}^{t'd-i-1 } q^{\rho'_2+i + 1}(x_1 \otimes g^{i+j}) \epsilon_{td-i, t'd-i-1}^* ,&\\
&[(x_1 \otimes g^j) \epsilon_{t''d-j+1, t'''d-j}^*, (x_1 \otimes g^i) \epsilon_{td-i+1,t'd-i}^*] & \\
& \hspace{30 mm} = ((t+t'')d - i -j) (x_1 \otimes g^{i+j}) \epsilon_{(t+t'')d-i-j+1, (t'+t''')d-i-j}^*,&\\
&[(x_2 \otimes g^j) \epsilon_{t'd-j, t''d-j+1}^*, (x_1 \otimes g^i) \epsilon_{td-i+1,t'd-i}^*] & \\
& \hspace{30 mm} = (t''d-j)(x_2 \otimes g^{i+j}) \epsilon_{(t+t'')d-i-j, (t'+t''')d-i-j+1}^*&\\
& \hspace{35 mm} -(t'd-i)q^i(x_1 \otimes g^{i+j}) \epsilon_{(t+t'')d-i-j+1, (t'+t''')d-i-j}^* ,&\\
&[(x_1 x_2 \otimes g^j) \epsilon_{t'd-j+1, t''d-j+1}^*, (x_1 \otimes g^i) \epsilon_{td-i+1,t'd-i}^*] &\\
&\hspace{10 mm} = (t''d-j+1)(x_1 x_2 \otimes g^{i+j}) \epsilon_{(t+t'')d-i-j+1, (t'+t''')d-i-j+1}^*&\\
& \hspace{15 mm} +(-1)^{td-i} \sum_{\rho''_1 =0}^{ td-i} q^{\rho''_1+1} (x_1 x_2 \otimes g^{i+j}) \epsilon_{(t+t'')d-i-j+1, (t'+t''')d-i-j+1}^* ,&\\
&[ (x_1 x_2 \otimes g^j) \epsilon_{0,0}^*,(x_1 \otimes g^i) \epsilon_{td-i+1,t'd-i}^*] &\\
&\hspace{30 mm}=(-1)^{td-i} \sum_{\rho''_1 = 0}^{td-i} q^{\rho''_1 + 1} (x_1 x_2 \otimes g^{i+j}) \epsilon_{td-i, t'd-i}^* ,&\\
&[(x_2 \otimes g^j) \epsilon_{t'd-j, t''d-j+1}^*, (x_2 \otimes g^i) \epsilon_{td-i,t'd-i+1}^*] & \\
& \hspace{30 mm} [(t'''d-j+1)q^j-(t'd-i+1)q^i] (x_2 \otimes g^{i+j}) \epsilon_{(t+t'')d-i-j, (t'+t''')d-i-j+1}^*,&\\
&[(x_1 x_2 \otimes g^j) \epsilon_{t'd-j+1, t''d-j+1}^*, (x_2 \otimes g^i) \epsilon_{td-i,t'd-i+1}^*]  &\\
&\hspace{10 mm} = (t'''d-j+1)q^j (x_1 x_2 \otimes g^{i+j}) \epsilon_{(t+t'')d-i-j+1, (t'+t''')d-i-j+1}^*&\\
& \hspace{15 mm} -(-1)^{td-i} \sum_{\rho''_2 =0}^{t'd-j} q^{\rho''_2-t''d+j+i-1} (x_1 x_2 \otimes g^{i+j}) \epsilon_{(t+t'')d-i-j+1, (t'+t''')d-i-j+1}^* ,\\
& \textrm{and}\\
&[(x_1 x_2 \otimes g^j) \epsilon_{0,0}^*, (x_2 \otimes g^i) \epsilon_{td-i,t'd-i+1}^*] &\\
&\hspace{30 mm} = (-1)^{td-i+1} \sum_{\rho'_2=0}^{t'd-i} q^{\rho'_2+i} (x_1 x_2 \otimes g^{i+j}) \epsilon_{td-i, t'd-i}^*. &
\end{flalign*}

We can compare these results to the computations in \cite[Section~5.3]{GNW} and \cite[Section~3.1]{BGMS} when $i=0$.

\subsection{$d>2$ is even}

If $d>2$ is even, then 
\begin{align*}
C_{g^i} =& \{\gamma \in (\N \cup \{-1\})^2 | ~\gamma_1 =-1 \textrm{ or } (\exists~ t, t' \in \N, \gamma_2+i = t\frac{d}{2} \textrm{ and } |\gamma| = t' \\
& \hspace{55 mm} \textrm{ and } t+t' \textrm{ even}) \textrm{ and }\\
& \hspace{35 mm} \gamma_2 =-1 \textrm{ or } (\exists~ t, t' \in \N, \gamma_1+i = t\frac{d}{2} \textrm{ and } |\gamma| = t' \\
& \hspace{55 mm}  \textrm{ and } t+t' \textrm{ even}) \}\\
=& \{\gamma \in (\N \cup \{-1\})^2 | ~\gamma_1 =-1 \textrm{ and } \gamma_2=-1 \textrm{ or }\\
& \hspace{35 mm} \gamma_1 =-1 \textrm{ and } \exists~ t \in \N, \gamma_2 = t'+1 \textrm{ and } \\
& \hspace{65 mm} i=t\frac{d}{2}+ 1 \textrm{ and } t+t' \textrm{ even } \textrm{ or } \\
& \hspace{35 mm} \gamma_2 =-1 \textrm{ and } \exists~ t \in \N, \gamma_1 = t'+1 \textrm{ and } \\
& \hspace{65 mm} i=t\frac{d}{2}+ 1 \textrm{ and } t+t' \textrm{ even } \textrm{ or } \\
& \hspace{35 mm} \exists~ t, t' \in \N, \gamma_1 = t\frac{d}{2}-i \textrm{ and } \gamma_2 = t'\frac{d}{2}-i \textrm{ and } t+t' \textrm{ even}\}. \\
\end{align*}

From this description, we immediately get  
\begin{align*} HH^*(\Lambda_{\q}^2, \Lambda_{\q}^2 \rtimes G) \cong & \bigoplus_{t \in \N} span_k \{(x_1 \otimes g) \epsilon_{0,2t+1}^*, (x_1 \otimes g^{\frac{d}{2}+1}) \epsilon_{0,2t}^*, (x_2 \otimes g) \epsilon_{2t+1, 0}^*, \\
& \hspace{10 mm} (x_2 \otimes g^{\frac{d}{2}+1}) \epsilon_{2t, 0}^*, (x_1 x_2 \otimes g)\epsilon_{2t, 0}^*, (x_1 x_2 \otimes g^{\frac{d}{2}+1})\epsilon_{2t+1, 0}^*, \\
& \hspace{10 mm} (x_1 x_2 \otimes g) \epsilon_{0,2t}^*, (x_1 x_2 \otimes g^{\frac{d}{2}+1}) \epsilon_{0,2t+1}^* \} \\
&\bigoplus_{i=1}^d \bigoplus_{\substack{ t, t' \in \N \\ t+ t' \textrm{ even}}} span_k\{ (1 \otimes g^i) \epsilon_{t\frac{d}{2}-i, t'\frac{d}{2}-i}^*, (x_1 \otimes g^i) \epsilon_{t\frac{d}{2}-i+1, t'\frac{d}{2}-i}^*, \\
& \hspace{20 mm} (x_2 \otimes g^i) \epsilon_{t\frac{d}{2}-i, t'\frac{d}{2}-i+1}^*, (x_1 x_2 \otimes g^i) \epsilon_{t\frac{d}{2}-i+1, t'\frac{d}{2}-i+1}^* \} \\
& \bigoplus_{i=1}^d span_k \{ (x_1 x_2 \otimes g^i)\epsilon_{0,0}^*, (1 \otimes g^i) \epsilon_{0,0}^* \} \\
\end{align*} as a vector space.

Then, after taking $G$-invariants, we get
\begin{align*} HH^*(\Lambda_{\q}^2 \rtimes G) \cong & \bigoplus_{t \in \N} span_k \{(x_1 \otimes g) \epsilon_{0,td-1}^*, (x_1 \otimes g^{\frac{d}{2}-1}) \epsilon_{0,td}^*, (x_2 \otimes g) \epsilon_{td-1, 0}^*, \\
& \hspace{10 mm} (x_2 \otimes g^{\frac{d}{2}+1}) \epsilon_{td, 0}^*, (x_1 x_2 \otimes g)\epsilon_{td, 0}^*, (x_1 x_2 \otimes g^{\frac{d}{2}+1})\epsilon_{td-1, 0}^*, \\
& \hspace{10 mm} (x_1 x_2 \otimes g) \epsilon_{0,td}^*, (x_1 x_2 \otimes g^{\frac{d}{2}+1}) \epsilon_{0,td-1}^* \} \\
&\bigoplus_{i=1}^d \bigoplus_{\substack{ t, t' \in \N \\ t+ t' \textrm{ even}}} span_k\{ (1 \otimes g^i) \epsilon_{t\frac{d}{2}-i, t'\frac{d}{2}-i}^*, (x_1 \otimes g^i) \epsilon_{t\frac{d}{2}-i+1, t'\frac{d}{2}-i}^*, \\
& \hspace{20 mm} (x_2 \otimes g^i) \epsilon_{t\frac{d}{2}-i, t'\frac{d}{2}-i+1}^*, (x_1 x_2 \otimes g^i) \epsilon_{t\frac{d}{2}-i+1, t'\frac{d}{2}-i+1}^* \} \\
& \bigoplus_{i=1}^d span_k \{ (x_1 x_2 \otimes g^i)\epsilon_{0,0}^*, (1 \otimes g^i) \epsilon_{0,0}^* \} \\
\end{align*} as a vector space.

\subsection{$q=-1$}

Now assume $q=-1$.  That is, $d=2$ and we are in the commutative truncated polynomial case.  Then 
\begin{align*}
C_{g^i} =& \{\gamma \in (\N \cup \{-1\})^2 | ~\gamma_1 =-1 \textrm{ or } \exists~ t \in \N, \gamma_1-i = 2t \textrm{ and } \\
& \hspace{35 mm} \gamma_2 =-1 \textrm{ or } \exists~ t \in \N, \gamma_2-i = 2t \}\\
=& \{\gamma \in (\N \cup \{-1\})^2 | ~\gamma_1 =-1 \textrm{ and } \gamma_2=-1 \textrm{ or }\\
& \hspace{35 mm} \gamma_1 =-1 \textrm{ and } \exists~ t \in \N, \gamma_2 = 2t+i \textrm{ or } \\
& \hspace{35 mm} \gamma_2 =-1 \textrm{ and } \exists~ t \in \N, \gamma_1 = 2t+i \textrm{ or } \\
& \hspace{35 mm} \exists~ t, t' \in \N, \gamma_1 = 2t+i \textrm{ and } \gamma_2 = 2t'+i \}. \\
\end{align*}

From this description, we immediately get  
\begin{align*} HH^*(\Lambda_{\q}^2, \Lambda_{\q}^2 \rtimes G) \cong & \bigoplus_{i=1}^2 \bigoplus_{t \in \N} span_k \{(x_1 \otimes g^i) \epsilon_{0,2t+i}^*, (x_2 \otimes g^i) \epsilon_{2t+i, 0}^*, (x_1 x_2 \otimes g^i)\epsilon_{2t+i+1, 0}^*, \\
& \hspace{10 mm} (x_1 x_2 \otimes g^i) \epsilon_{0,2t+i+1}^* \} \\
&\bigoplus_{i=1}^2 \bigoplus_{t, t' \in \N} span_k\{ (1 \otimes g^i) \epsilon_{2t+i, 2t'+i}^*, (x_1 \otimes g^i) \epsilon_{2t+i+1, 2t'+i}^*, \\
& \hspace{20 mm} (x_2 \otimes g^i) \epsilon_{2t+i, 2t'+i+1}^*, (x_1 x_2 \otimes g^i) \epsilon_{2t+i+1, 2t'+i+1}^* \} \\
& \bigoplus_{i=1}^d span_k \{ (x_1 x_2 \otimes g^i)\epsilon_{0,0}^*, (1 \otimes g^i) \epsilon_{0,0}^* \} \\
\end{align*} as a vector space.

Then, after taking $G$-invariants, we get
\begin{align*} HH^*(\Lambda_{\q}^2 \rtimes G) \cong & \bigoplus_{t \in \N} span_k \{(x_1 \otimes g) \epsilon_{0,2t+1}^*, (x_2 \otimes g) \epsilon_{2t+1, 0}^*, (x_1 x_2 \otimes g)\epsilon_{2t, 0}^*, \\
& \hspace{10 mm} (x_1 x_2 \otimes g) \epsilon_{0,2t}^* \} \\
&\bigoplus_{i=1}^2 \bigoplus_{t, t' \in \N} span_k\{ (1 \otimes g^i) \epsilon_{2t+i, 2t'+i}^*, (x_1 \otimes g^i) \epsilon_{2t+i+1, 2t'+i}^*, \\
& \hspace{20 mm} (x_2 \otimes g^i) \epsilon_{2t+i, 2t'+i+1}^*, (x_1 x_2 \otimes g^i) \epsilon_{2t+i+1, 2t'+i+1}^* \} \\
& \bigoplus_{i=1}^d span_k \{ (x_1 x_2 \otimes g^i)\epsilon_{0,0}^*, (1 \otimes g^i) \epsilon_{0,0}^* \} \\
\end{align*} as a vector space.

\subsection{$q=1$}

Finally, if $q=1$, we are considering the truncated skew polynomial ring.  In this case, 
\begin{align*}
C_{1} =& \{\gamma \in (\N \cup \{-1\})^2 | ~\gamma_1 =-1 \textrm{ or } \exists~ t \in \N, \gamma_1+\gamma_2 = 2t \textrm{ and } \\
& \hspace{35 mm} \gamma_2 =-1 \textrm{ or } \exists~ t \in \N, \gamma_1+\gamma_2 = 2t \}\\
=& \{\gamma \in (\N \cup \{-1\})^2 | ~\gamma_1 =-1 \textrm{ and } \gamma_2=-1 \textrm{ or }\\
& \hspace{35 mm} \gamma_1 =-1 \textrm{ and } \exists~ t \in \N, \gamma_2 = 2t+1 \textrm{ or } \\
& \hspace{35 mm} \gamma_2 =-1 \textrm{ and } \exists~ t \in \N, \gamma_1 = 2t+1 \textrm{ or } \\
& \hspace{35 mm} \exists~ t \in \N, \gamma_1 +\gamma_2= 2t \}. \\
\end{align*}

From this description, we immediately get  
\begin{align*} HH^*(\Lambda_{\q}^2 \rtimes G) \cong & \bigoplus_{t \in \N} span_k \{(x_1 \otimes 1) \epsilon_{0,2t+1}^*, (x_2 \otimes 1) \epsilon_{2t+1, 0}^*, (x_1 x_2 \otimes 1)\epsilon_{2t, 0}^*, \\
& \hspace{10 mm} (x_1 x_2 \otimes 1) \epsilon_{0,2t}^* \} \\
& \bigoplus_{\substack{t, t' \in \N \\ t+t' \textrm{ even }}} span_k\{ (1 \otimes 1) \epsilon_{t, t'}^*, (x_1 \otimes 1) \epsilon_{t+1, t'}^*, (x_2 \otimes 1) \epsilon_{t, t'+1}^*, \\
& \hspace{10 mm} (x_1 x_2 \otimes 1) \epsilon_{t+1, t'+1}^* \} \\
& \bigoplus span_k \{ (x_1 x_2 \otimes 1)\epsilon_{0,0}^*, (1 \otimes 1)\epsilon_{0,0}^* \} \\
\end{align*} as a vector space.

Because the group action is trivial in this case, we can directly compare this result to \cite[Section~3.5]{BGMS} and \cite[Section~5.6]{GNW}.


\section{Acknowledgement}
I would like to thank Sarah Witherspoon for her guidance throughout this work and in the preparation of this document.


\end{document}